\documentclass[12pt]{amsart}
\usepackage{extsizes}
\usepackage{hyperref}
\usepackage{mathtools}
\usepackage{comment}
\hypersetup{
    colorlinks=true,
    linkcolor=blue,
    filecolor=magenta,      
    urlcolor=cyan,
}

\usepackage{graphicx}
\usepackage{multicol}

\usepackage{tikz-cd}

\usepackage{comment}

\usepackage{xpatch}
\makeatletter
\g@addto@macro\th@remark{\thm@headpunct{}}
\makeatother

\newtheorem{theor}{\hspace{1cm}{\sc Theorem}}[section]
\newtheorem{utver}[theor]{\hspace{1cm}{\sc Proposition}}

\newtheorem{lemma}[theor]{\hspace{1cm}{\sc Lemma}}

\newtheorem*{utver*}{\hspace{1cm}{\sc Proposition}}
\theoremstyle{definition}

\newtheorem{defin}[theor]{\hspace{1cm}{\sc Definition}}
\newtheorem{exa}[theor]{\hspace{1cm}{\sc Example}}

\DeclareMathOperator{\supp}{supp}
\DeclareMathOperator{\cod}{codim}
\DeclareMathOperator{\rk}{rk}
\DeclareMathOperator{\im}{Im}

\DeclareMathOperator{\conv}{conv}
\DeclareMathOperator{\mv}{MV}

\DeclareMathOperator{\inter}{int}

\usepackage{comment}
\usepackage{graphicx}
\usepackage{subcaption}

\usepackage{amssymb}

\usepackage{geometry}

\numberwithin{equation}{section}

\begin{document}

\title{Effective membership problem and systems of polynomial equations with multiple roots.
}


\author{I. Nikitin 
}

\footnotetext{
              University of Toronto, Department of Mathematics; \\
              \email{vanya.nikitin@mail.utoronto.ca}           
}


\maketitle

\begin{abstract}
For a tuple of $k+1$ convex polytopes $(A, B,\ldots, B)$ we solve the so-called effective membership problem, i.e. for a tuple $f=(f_1,\ldots, f_k)$ of polynomials satisfying some certain properties of generality and having Newton polytope $B$ we provide a method to compute $\mathbb{C}^A\cap I$, where $I$ is an ideal in the ring of Laurent polynomials generated by $f$. We connect this topic to the study of mixed discriminants and multiple solutions of systems of polynomial equations.
\end{abstract}

\section{\bf Introduction.}
\label{intro}

\subsection{Nullstellensatz and membership problem.}

\indent

\indent

Classical Hilbert Nullstellensatz says that the system of polynomial equations $f_i=0,\ i=1,\ldots, k$ over algebraically closed field (say $\mathbb{C}$) has no solutions if and only if the ideal generated by the tuple $(f_1,\ldots, f_k)$ contains unity, that is, there exists a decomposition$\colon$
\begin{equation}\label{eq1}
1=c_1 f_1+\ldots+c_k f_k,\ c_i\in\mathbb{C}[\bf{x}].
\end{equation}

 To find $c_i$ from \ref{eq1} we would like to have some upper bound imposed on the degrees of $c_i$. The so-called {\it{Effective Nullstellensatz}} is a branch of results dedicated to the problem of finding the multiples in \ref{eq1} with the least possible degree for a given type of systems $f_1,\ldots, f_k$. For example, we have the well known result \cite{B} of W. Dale Brownawell published in 1987$\colon$

\begin{theor}[W. Dale Brownawell, 1987]

Let the polynomials $P_1,\ldots, P_m\in\mathbb{C}[\bf{x}]$ with $\deg P_i\leq D$ have no common zero in $\mathbb{C}^n$. Then there are polynomials $A_1,\ldots, A_m\in\mathbb{C}[\bf{x}]$ with

$$
\deg A_i\leq\mu n D^\mu+\mu D,
$$
where $\mu=\min\{m, n\}$, such that 
$$
1=A_1 P_1+\ldots+A_m P_m.
$$
\end{theor}

In 1988, one year later in \cite{K} Kollar obtained a new bound for polynomials of degree at least $3$. The estimate given in the paper is sharp in the following meaning$\colon$If we have a system $f_i=0,\ d_i=\deg f_i\geq3$ with no common zeroes then there exists a tuple of polynomials $c_1,\ldots, c_k$ such that $1=f_1 c_1+\ldots+f_k c_k$, $\deg f_i c_i\leq N(d_1,\ldots, d_k, k, n)$, where $N(d_1,\ldots, d_k, k, n)$ is a function depending only on $d_i, k, n$. Moreover there exists a system $f^0_1,\ldots, f^0_k$ with no common solutions such that there is NO set of multiples $c_i$ with the property $1=c_1 f^0_1+\ldots+c_k f^0_k$ such that $\deg f_i c_i<N(d_1,\ldots, d_k, k, n)$ for at least one of $i$. Other results on the Effective Nullstellensatz were obtained in \cite{SS}, \cite{S} and \cite{J}, for example, and recently \cite{AJ} which enhances theorem \ref{Tui} below.

The estimate can be significantly improved if we consider the set of sufficiently generic systems. The first step toward this is to consider the ring of Laurent polynomials $\mathbb{C}[\bf{x},\bf{x}^{-1}]$ instead of the ordinary ring of polynomials. For example we have the following result from \cite{CE}$\colon$

\begin{theor}[Canny, J., and I. Z. Emiris., 2000]

Let $A_0,\ldots, A_n\subset\mathbb{Z}^n$ be convex polytopes. Let $A=A_0+\ldots+A_n$ be of full dimension $n$. Then there exists a Zariski open subset $U$ of the space $\mathbb{C}^{A_0}\oplus\ldots\oplus\mathbb{C}^{A_n}$ such that for tuple of polynomials $f=(f_0,\ldots, f_n)\in U$ and any $g\in\mathbb{C}^{A}$ there exists a decomposition $g=c_0 f_0+\ldots+c_n f_n$ with $\supp c_i\subset\sum A_j,\ j\neq i$.

\end{theor}

From this result we can obtain a new version of Nullstellensatz putting $A_i=d_i E^n$ where $E^n$ is the standard n dimensional simplex and $g=1$. 

In paper \cite{T} the following was obtained$\colon$ 

\begin{theor}[J. Tuitman, 2011]\label{Tui}

    Let $A_0,\ldots, A_n\subset\mathbb{Z}^n$ be convex polytopes and $f=(f_1,\ldots, f_n)\in\mathbb{C}^{A_0}\oplus\ldots\oplus\mathbb{C}^{A_n}$. Let $A=A_0+\ldots+A_n$ be of full dimension $n$. Then $f$ is non-degenerate if and only if for any $g\in\mathbb{C}^{A}$ there exists a decomposition $g=c_0 f_0+\ldots+c_n f_n$ with $\supp c_i\subset\sum A_j,\ j\neq i$.
\end{theor}

This theorem, as well as the classical paper \cite{BR} uses the same definition of non-degeneracy. We will remind it now. For a given polytope $A\subset\mathbb{R}^n$ and $v\in\mathbb{R}^n$ we denote $A(v)$ to be the subset $\{a\vert v.a=\min_{x\in A}{v.x} \}$ of $A$. For a given polynomial $f=\sum f_p \bf{x}^p$ and a subset $A\subset\mathbb{Z}^n$ we denote by $f\vert_A$ the new polynomial $\sum f_p\bf{x}^p$, where the sum is only taken for $p\in A\cap\supp f$. We call the tuple of equations $f_i=0$ {\it non-degenerate} with respect to the set of polytopes $A_i$ if the corresponding system $f_i\vert_{A_i(v)}=0$ has no solutions for any direction $v\neq0$.

We see that all results formulated above allow us to reduce the problem of finding the multiples in the decomposition \ref{eq1} to the problem of solving a system of linear equations. 

Let us introduce another problem closely related to the topic. Let $(A, B_1,\ldots, B_k) = (A, B)$ be a tuple of subsets of $\mathbb{Z}^n$. For a tuple of polynomials $f=(f_1,\ldots, f_k)\in\mathbb{C}^{B_1}\oplus\ldots\oplus\mathbb{C}^{B_k}=\mathbb{C}^B$, we consider the ideal generated by $f$ in the ring of Laurent polynomials. The question is what we can say about the set of polynomials that are members of the ideal $(f)$ and have support in $A$. We denote$\colon$
$$
\mathcal{V}_{A}^{f}=(f) \cap\mathbb{C}^{A}.
$$
\begin{exa}\label{ex1}

Let $f=1+x\in\mathbb{C}[x]$ and $A=\{0,2\}$. We can easily compute that $\mathcal{V}_{A}^{f}=\{c(1-x^2),\ c\in\mathbb{C}\}$.
\end{exa}
It is clear that $\mathcal{V}_{A}^{f}$ depends on the choice of $f$, we will study the special case when $f$ satisfies some conditions of generality. 
\begin{defin}
    We will say that $f=(f_1,\ldots, f_k)\in\mathbb{C}^B$ is r-general if
    \begin{enumerate}
        \item $f_{\sigma(1)},\ldots, f_{\sigma(k)}$ is a regular sequence in the ring of Laurent polynomials for any permutation $\sigma\in S_k$.
        \item $f_{\sigma(1)}\vert_F,\ldots, f_{\sigma(k)}\vert_F$ is a regular sequence in the ring of Laurent polynomials for each facet $F$ of $B$ and for any permutation $\sigma\in S_k$.
        \item The Newton polytope of $f_i$ is $B$.
    \end{enumerate}
\end{defin}

Alongside {\it r-general} we will use the term {\it generic} when its clear from the context what kind of generality we assume, otherwise we will give explanations.

In general $\mathcal{V}_{A}^{f}$ is always a finite dimensional vector space because $\mathcal{V}_{A}^{f}\subset\mathbb{C}^A$ as a linear subspace and $|A|$ is finite. We arrive at the problem how to compute $\dim\mathcal{V}_{A}^{f}$ in terms of $A, B$ for r-general $f\in\mathbb{C}^{B}$. For example we have the following simple fact$\colon$

\begin{utver*}\label{prop1}

If $A$ is a convex polytope in $\mathbb{Z}^n$ and $k=1$ then $\dim\mathcal{V}_{A}^{f}=|A\ominus B|$ for $f$ with Newton polytope $B$. Here $A\ominus B=\{b\in\mathbb{Z}^n\vert b+B\subset A\}$ and it has a special name$\colon$
\end{utver*}

\begin{defin}
    For two subsets $X, Y\subset\mathbb{R}^n$, we denote $\{s\in\mathbb{R}^n\vert s+X\subset Y\}$ by $Y\ominus X$. It is called the Pontryagin difference.
\end{defin}

If $A$ is not convex (see example \ref{ex1}) then of course $\mathcal{V}_{A}^{f}\neq\mathbb{C}^{A\ominus B}$ in general. But it is still true that $\mathcal{V}_{A}^{f}$ is a subspace of $\mathbb{C}^{\conv(A)\ominus B}$. For $k>1$ the situation is getting more interesting.

\begin{exa}\label{ex2}

Consider the tuple $(A, A\cup\{q\}, A\cup\{q\})$. Then a pair of polynomials supported at $A\cup\{q\}$ can be written in the form $f_1=f^0_1+\alpha_1{\bf{x}}^q,\ f_2=f^0_2+\alpha_2{\bf{x}}^q$, where $f^0_1$ and $f^0_2$ are supported at $A$. We can see that $f_1\alpha_2-f_2\alpha_1\in\mathbb{C}^{A}$ so $\mathcal{V}^{f}_A$ is non-empty, but $A\ominus(A\cup\{q\}))=A\ominus(A\cup\{q\}))=\varnothing$.    
\end{exa}

We now give a definitions strongly related to one we have already seen. Let as before $f$ be a tuple of polynomials $(f_1,\ldots, f_{k})$. For $X\subset\mathbb{R}^n$ we denote $X\cap\mathbb{Z}^n$ by $\mathbb{Z}(X)$.

\begin{defin}

    Let $C=(C_1,\ldots, C_{k})$, $C_i\subset\mathbb{Z}^n$. We denote by $\mathcal{V}^{C, f}_A$ the following space$\colon$

$$
\mathcal{V}^{C,\ f}_A=\{g\in\mathbb{C}[{\bf{x}},{\bf{x}^{-1}}]\vert\ \exists (c_1,\ldots, c_{k})\ \text{s.t.}\ g=c_1 f_1+\ldots+c_{k}f_{k}\ \text{and}\ c_i\in\mathbb{C}^{C_i} \}\cap\mathbb{C}^{A}
$$
\end{defin}

The difference between $\mathcal{V}^{C,\ f}_A$ and $\mathcal{V}^{f}_A$ is that $\mathcal{V}_{A}^{f}$ is a space of all members of the ideal generated by $(f)$ with given support $A$ and $\mathcal{V}^{C,\ f}_A$ is a space of all members who can be expressed in the form $c_1 f_1+\ldots+c_{k}f_{k}$ and all the supports of $c_i$ are also restricted by given subsets $C_i$. 

\indent

It is easy to see that $\mathcal{V}^{f}_A=\mathcal{V}^{\mathbb{Z}^n,\ f}_A$

\indent

Since $\dim\mathcal{V}_A^{f}$ is finite, it is easy to see that there exists a compact $C^0=(C^0_1,\ldots, C^0_{k})$, $C_i\subset\mathbb{R}^n$ s.t. $\mathcal{V}^{C^0, f}_A=\mathcal{V}^{f}_A$. We will also use that obviously $\dim\mathcal{V}_A^{f}=\dim\mathcal{V}_{A+u}^{f x^v}$ and $\dim\mathcal{V}_A^{C, f}=\dim\mathcal{V}_{A+u}^{C+u-v,\ f x^v}$ for any $u,\ v\in\mathbb{Z}^n$. The aim of effective membership problem is to find such $C^0$ for prescribed tuple of polynomials $(f)$. 

Let us formulate in details the main result proven in the section \ref{ga3}$\colon$

\begin{theor}\label{THM}
    Let $A, B$ be polytopes in $\mathbb{Z}^n$ such that the last one is not a segment. Let $A+v\subset(t+1)B$ for some non-negative real $t$ and $v\in\mathbb{Z}^n$. Then, for any r-general tuple of polynomials $f_1,\ldots, f_k$ supported at $B$ and any combination $g=c_1 f_1+\ldots+c_k f_k$ supported at $A$, we can represent $g$ in the form $g=s_1 f_1+\ldots+s_k f_k$ where all $s_i$ are supported at $\mathbb{Z}(tB)$.
\end{theor}

\begin{exa}
    For $A=B$, we have $\dim\mathcal{V}_A^{f}=k$ for r-general $f\in\mathbb{C}^B$. And for $A=LB$, where $L$ is not necessarily integral, $$\dim\mathcal{V}_A^{f}=\sum^{k}_{i=1,\ i\leq L}(-1)^{i-1}\binom{k}{i}|\mathbb{Z}((L-i)B)|.$$ as it follows from the proof of Theorem \ref{THM}. 
\end{exa}

The monotonicity of $\dim\mathcal{V}_A^f$ and the proof of theorem \ref{THM} imply that that if $L_1 B\subset A\subset L_2 B$ then 

$$
\sum^{k}_{i=1,\ i\leq L}(-1)^{i-1}\binom{k}{i}|\mathbb{Z}((L_1-i)B)|\leq\dim\mathcal{V}_A^{f}\leq\sum^{k}_{i=1,\ i\leq L}(-1)^{i-1}\binom{k}{i}|\mathbb{Z}((L_2-i)B)|.
$$

Though we cannot give the precise formula for $\dim\mathcal{V}_A^f$ when $A$ is not a Homothety of $B$. 

Results similar to \ref{THM} were obtained in \cite{W}, \cite{CDV} (Theorem 3) in somewhat different settings. We generalise Theorem 1.2. from \cite{W}.

\begin{theor}\label{Wulcan}[E. Wulcan, 2011]
Let $F_1,\ldots, F_m$, and $\Phi$ be polynomials in $\mathbb{C}[z_1,\ldots, z_n]$ and let $P$ be a smooth and “large” polytope that contains the origin and the support of $\Phi$ and the coordinate functions $z_1,\ldots, z_n$. Assume that $\Phi\in(F_1,\ldots, F_m)$ and moreover that the codimension of the zero set of the $F_j$ is $m$ and that it has no component contained in the variety at infinity. Then there are polynomials $G_j$ such that

$$
\sum F_i G_i=\Phi
$$
and $$\supp(F_j G_j )\subset P.$$
\end{theor}
Descriptions of `large' and `has no components at infinity' are given in the original paper. The main difference between \ref{Wulcan} and \ref{THM} is that we do not assume smoothness of the polytope and also that we work in the ring of Laurent polynomials. A polytope of any dimension $n$ is called smooth if each vertex is contained in precisely $n$ edges and the set of corresponding $n$ primitive vectors forms a lattice basis.

To demonstrate the main idea of the proof of theorem \ref{THM}, in section \ref{mex} we will consider the simplified construction similar to what we are going to do later, in section \ref{ga3}, in more general settings and, at last, in section \ref{MR}.

\subsection{Study of discriminants.}

\indent 

\indent

The space $\mathcal{V}^{f}_{A}$ defined above plays an important role in the study of discriminant. For the convenience of the reader let us give some basic explanations. The discriminant of a polynomial with the support $A$ is the closure of the set $\triangle_{A}=\{f\in\mathbb{C}^{A}|\ df=f=0\ \text{has a solution}\}$. The discriminant variety $\triangle_{A}=0$ is projectively dual to the toric variety $X_{A}$ \cite{GKZ} (p.3), that in particular can be defined as follows$\colon X_A$ is the closure of the image of $(\mathbb{C}^{*})^n$ by the so-called monomial map $m_{A}$ 
$$
m_{A}\colon(\mathbb{C}^{*})^{n}\rightarrow\mathbb{CP}^{A}
$$
$$
x\mapsto [\ldots\colon x^{a_{i}} \colon\ldots].
$$

The notion of the discriminant can also be applied to systems of several polynomial equations. For $k=n$ i.e. the case when the number of equations and variables are equal to each other, we have the following classical result$\colon$

\begin{theor}{Bernstein–Kushnirenko formula}

    For a generic system of polynomial equations $f_1=\ldots=f_n$ in $n$ complex variable the number of its solutions on the complex torus $(\mathbb{C}^*)^n$ is equal to $n!\mv(A_1,\ldots, A_n)$ where $A_i$ is the Newton polytope of the corresponding equation and $\mv$ is the mixed volume.  
\end{theor}

A system has a multiple solution if at some point $p$ we have $f(p)=0$ and $df_1(p),\ldots, df_n(p)$ are linearly dependent. The set of all such systems lies in an algebraic subset of the space of parameters $\mathbb{C}^A=\mathbb{C}^{A_1}\oplus\ldots\oplus\mathbb{C}^{A_n}$.   

We define the {\it discriminantal variety} of a tuple $A=(A_1,\ldots, A_k)$ as the closure in $\mathbb{C}^A$ of the set of all systems with multiple solution. 

Discriminants of a general polynomial of several variables with a prescribed Newton polytope were first systematically studied in \cite{GKZ}. The generalization to two or more general polynomial equations is studied starting from \cite{EMD} and \cite{MD}.

In the work \cite{E} it was shown that for sufficiently good $A$ a generic point on the discriminant corresponds to the system with the unique degenerate root which has multiplicity 2. By sufficiently good we mean the following characterisations$\colon$

\begin{defin}
    A tuple of finite sets $X_1,\ldots, X_k$ in $Z^n$ is said to be reduced, if they
cannot be shifted to the same proper sublattice of $Z^n$.
    
\end{defin}

For most of $A$ there are systems that have a root of higher multiplicity. Let us denote by $S_k$ the set of all systems that have a root of multiplicity $k$. We expect that $S_{k+1}$ has codimension 1 in the previous one. Natural problems that arise here are$\colon$

\begin{enumerate}
    \item Classify all $A$ for which $S_{k}$ is empty for some low $k$.
    \item What properties should $A$ satisfy to make $S_{k}$ of expected codimension?
    \item Estimate the length of the maximal possible chain of non-empty consecutive strata.
\end{enumerate}
    $$S_{k}, S_{k-1}, \ldots, S_{0}.$$
    
For example, if $\mv(A)=2$ the strata $S_k$ will be empty for $k\geq3$. The third question is not trivial since there exist tuples with a system who has root of multiplicity $l$ and with a system who has root of multiplicity $l+2$ but for this tuple there is no system with multiplicity $l+1$. For example take bodies depicted below on fig.\ref{exf2}. A system supported on the following pair of bodies have no root of multiplicity $5$, though it can have a root of multiplicity $4$ and $6$.

\begin{figure}[h!]
\includegraphics[width=0.3\linewidth]{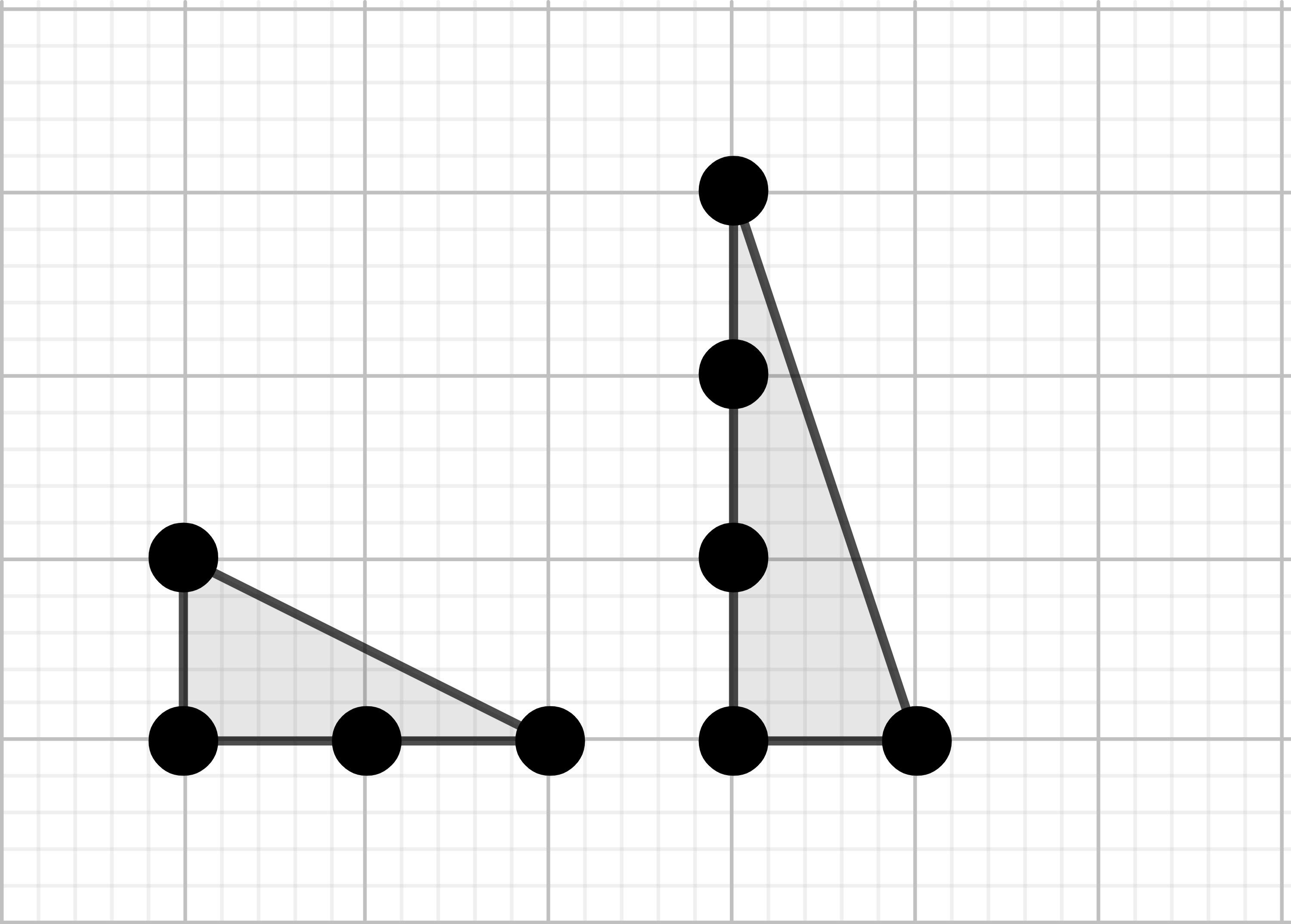}
\caption{tuple with a discontinues chain of strata}
\label{exf2}
\end{figure}

For $k=2, n=d$ these problems were solved in \cite{BN} and \cite{E}, which lead to new results on the Galois groups of general systems of polynomial equations. From the results of the work \cite{GK} it follows that the maximal possible $k$ such that $S_k$ is non empty can be bounded from above by a function depending on the total number of monomials in $A$ but not on their degrees. In this direction some new interesting results were obtained in \cite{BDF}. Namely, for a system in $n$ variables of $n$ equations with the same support set $A$, the highest possible multiplicity of an isolated root was estimated in terms of the decomposition $|A|=n+m+1$. Also, in the recent work \cite{DRM} it was shown that the only smooth polygons $A$ such that curves $\{f=0\},\, f \in \mathbb{C}^A$, may have at most $A_1$-singularities, are equivalent to the 2–simplex or a unit square.

The readers familiar with the topic will likely be interested to compare our study to the classical question of dual defects for toric varieties. Indeed, if we have one support set $A$, then the space of polynomials $\mathbb{C}^A$ can be similarly split into strata $M_k$ according to the multiplicity of the isolated solution of the equation $df=0$ at a degenerate root of f (called the Milnor number of the respective singular point of $f=0$). In this setting, we have two options$\colon$ either $M_1$ is non-empty (the $A$-toric variety is non-defective), or all $M_k$ are empty (the $A$-toric variety is defective). This is similar to our setting of square systems of equations with irreducible support sets (either $S_1$ is non-empty, or all $S_k$ are empty), though the classification of defective $A$ is much more complicated than the respective results \cite{BN}, \cite{E} for square systems (see \cite{DFS}, \cite{E2}, \cite{FI} for three different characterizations and \cite{DC} for their comparison). Extending this analogy to higher $k$, one might expect that, similarly to our example of empty $S_k$ with non-empty $S_{k-1}$ and $S_{k+1}$, one can expect examples of empty $M_k$ with non-empty $M_{k-1}$ and $M_{k+1}$, and the absence of such gaps with small $k$ for sufficiently large support sets $A$ (similarly to Theorem \ref{THM}).

One of the aims of this paper is to bring together two topics described above. We are ready to show the meaning of $\mathcal{V}_A^f$ to the studying of discriminants. We begin with a well known result from the differential geometry of curves$\colon$

\begin{theor}\label{DGT1}

\indent

Let $\gamma$ be a projective curve that is given locally (in some neighbourhood $U$ of a generic point $p\in\gamma$) by a family of holomorphic functions$\colon$ 
\[
\phi\colon U\rightarrow\mathbb{P}^n
\]
\[
z\mapsto[\ldots\colon \phi_{i}(z)\colon\ldots].
\]
If the dimension of the projective hull of $F$ equals $k$, then there exists a descending chain of pencils of osculating hyperplanes
$$
V^{0}\supset V^{1}\supset V^{2}\supset\ldots\supset V^{k},
$$
such that hyperplanes lying in $V^{i}$ intersect $F$ at p with multiplicity at least $i$, and $\dim V^{i}-\dim V^{i-1}=1$. In fact, $V^{i}$ contains a dense subset of hyperplanes that intersect $\gamma$ with multiplicity exactly $i$ and a proper subspace of hyperplanes that intersect $\gamma$ with multiplicity at least $i+1$ for $i<k$, and $V^{k}$ contains a unique hyperplane that intersects $\gamma$ with multiplicity $k$.

\end{theor}

For the sake of readability, we will prove it in the next section, as a proposition. Now, we apply it to the system of $n$ polynomial equations with $n$ variables. Again, let us consider the monomial map$\colon$

$$
m_{A}\colon(\mathbb{C}^{*})^{n}\rightarrow\mathbb{CP}^{A}
$$
$$
x\mapsto [\ldots\colon x^{a_{i}} \colon\ldots].
$$

Consider a system $f=\{f_1=\ldots=f_{n-1}=0\}$ such that all polynomials are from $\mathbb{C}^B$. For generic $f$ this system defines a smooth curve $F$ on the complex torus $(\mathbb{C}^*)^n$ by \cite{Kh}.

The map $m_A$ is well defined on $(\mathbb{C}^{*})^n$. We denote $m_{A}((\mathbb{C}^{*})^n)$ by $\mathbb{T}$. 
Let $g=\sum_{a\in A}c_{a}x^{a}$ be a polynomial supported at $A$. By $\pi_{g}$ we denote the hyperplane in $\mathbb{CP}^{A}$ given by the linear form $\sum_{a\in A} c_{a}z_{a}$, i.e. $\pi_{g}$ has the same coefficients as $g$ has. Assume that the projective hull of $m_A(F)$ has dimension $k$. Then, according to the theorem \ref{DGT1} at the generic point $p$ of $m_A(F)$ we have a set of hyperplanes $H_i$, $i=0,\ldots, k$ in $\mathbb{CP}^{A}$ such that $\pi_i$ intersects $m_A(F)$ at $p$ with multiplicity $i$.

In the next section we will prove that $k=|A|-\dim\mathcal{V}_A^{f}-1$ and it does not actually depend on the choice of a generic curve $F$ supported at $B$. Moreover the intersection multiplicity of $\pi_i$ and $m_A(F)$ at $p$ is equal to the intersection multiplicity of $F$ and the hypersurface $f_{n}^{i}$ which corresponds to the hyperplane $\{\pi_i=0\}$.

One of the main results given in this paper is$\colon$

\begin{theor}\label{th1}

\indent

Let $A, B_1,\ldots, B_{n-1}\subset\mathbb{Z}^{n}$ be a tuple of polytopes such that $\mv(A, B_1,\ldots, B_{n-1})>0$. Then for almost all points $p$ on a generic curve $\gamma=\{f=0\}$ supported at $B$, there exists a set of osculating curves $g_{i}$ supported at $A$, such that $g_{i}$ intersects $\gamma$ at $p$ with multiplicity $i$ for any $i\leq|A|-\dim\mathcal{V}_A^{f}-1$, i.e. the multiplicity of the root $p$ of the system $f=g_i=0$ is $i$.
\end{theor}

In particualr, as we will see later, for $A=LB$, $L\in\mathbb{Z}$, for generic tuple $f\in\mathbb{C}^B$ we can achieve $|A|-\dim\mathcal{V}_A^f-1=$ 

$$=\sum^{n-1}_{i=0,\ i\leq L}(-1)^{i}\binom{n-1}{i}|\mathbb{Z}((L-i)B)|-1=$$

which is the partial $(n-1)-$th finite difference of the Ehrhart polynomial of $B$, and thus it increases to infinity as $L^1$. Moreover, using the formula $\triangle^{n-1}_1 x^n=n! x-\frac{n!(n-1)}{2}$ for finite difference, and for $L\geq n-1$, we have 

$$
=n!V(B)L-V(B)\frac{n!(n-1)}{2}+(n-1)!\beta(B)-1=
$$

$$
=\mv(LB,B,\ldots,B)-V(B)\frac{n!(n-1)}{2}+(n-1)!\beta(B)-1.
$$

Where $V(B), \beta(B)$ are the lattice volume of $B$ and the second coefficient of the Ehrhart polynomial correspondingly. It's interesting to compare the maximal possible amount of isolated roots of the system suported at $(LB,B,\ldots,B)$ and the expression above. We can see that they vary by the term $\beta(B)(n-1)!-V(B)n!(n-1)/2-1$. From this we can deduce the inequality$\colon$

$$\beta(B)(n-1)!-V(B)n!(n-1)/2-1\leq 0$$

which is 

$$
\beta(B)\leq\frac{n(n-1)}{2}V(B)+1/(n-1)!.
$$
And this inequality is well known, for example see \cite{BDD}, Theorem 6.

In fact we can derive from theorem \ref{th1} here that the discriminantal variety of the tuple $(A, B)$ admits a chain of consecutive non empty strata of expected dimension. We will do it in section \ref{str}.

In sections \ref{NC} and \ref{pc} we are working on the technical tools designed to justify the major steps of the proof of Theorem \ref{THM}. In section \ref{ga3} we prove it and develop the theory of computing $\mathcal{V}_A^f$ in dimension $n$ for arbitrary $A$ in a particular case $k=3$, $B_1=B_2$. In the final section, we prove the main result by induction, using the particular case $k = 3$ to make an induction step.  

\section{\bf Proof of \ref{th1}.}\label{prel}
In this section we prove theorem \ref{th1} and \ref{DGT1}. Let us introduce some notation. Recall that for each $a=(a_{1},\ldots, a_{n})\in\mathbb{Z}^{n}$ we use the expression $z^{a}$ to denote the monomial $z_{1}^{a_{1}}\ldots z_{n}^{a_{n}}$, where $z_{i}\in\mathbb{C}^{*}$ and for each $A\subset\mathbb{Z}^{n}$ we denote by $\mathbb{C}^{A}$ the space of Laurent polynomials $\sum_{a\in A}c_{a}z^a$ supported at $A$. For any polynomial $f$ the convex hull of its support is called the Newton polytope of $f$. For a tuple of subsets $S=(A_{1},\ldots, A_{k})\subset\mathbb{Z}^{n}$, we have $\mathbb{C}^{S}=\mathbb{C}^{A_{1}}\oplus\ldots\oplus\mathbb{C}^{A_k}$. We identify each  $f=(f_{1},\ldots, f_{k})\in\mathbb{C}^{S}$ with the corresponding system of polynomial equations $f_{1}=\ldots=f_{k}=0$. For any two subsets $A, B$ of\ $\mathbb{Z}^{n}$ their Minkowski sum $A+B$ is defined as follows$\colon A+B=\{a+b|a\in A, b\in B\}\subset\mathbb{Z}^{n}$. We say that a finite subset $A\subset\mathbb{Z}^n$ is convex if it can be represented in the form $A=\mathbb{Z}^n\cap\bar{A}$, for some convex $\bar{A}\subset\mathbb{R}^{n}$.

As it was promised, we give here a proof of \ref{DGT1}

\begin{proof}
    Let the dimension of the projective hull of $\gamma$ be equal to $k$. Without loss of generality we may assume that $\gamma$ lies in the projective space of dimension $k$ and is not contained in any hyperplane. The latter implies that coordinate functions $\phi_{i}(z)$ are linearly independent. It is well known that any set of analytic functions is linearly independent if and only if the Wronskian of this set is not identically 0. Recall that the Wronskian can be computed by the following formula
\[\det
\begin{bmatrix}
    \phi_{0}       & \phi_{1} & \phi_{2} & \dots & \phi_{k} \\
    \phi_{0}^{(1)}       & \phi_{1}^{(1)} & \phi_{2}^{(1)} & \dots & \phi_{k}^{(1)} \\
    \hdotsfor{5} \\
    \phi_{0}^{(k)}       & \phi_{1}^{(k)} & \phi_{2}^{(k)} & \dots & \phi_{k}^{(k)}
\end{bmatrix}.
\]
\\If at some point of the curve the Wronskian vanishes, we will say that this point is Weierstrass, it can be shown that the property of being Weierstrass is invariant with respect to the choice of holomorphic parametrization. Now let us provide a relation between that property and the statement of the Lemma. Expand each $f_{i}$ as a Laurent power series at $p$ and substitute them into the expression $H(x_{0},\ldots, x_{k})=c_{0}x_{0}+c_{1}x_{1}+\ldots+c_{n}x_{n}$. Regrouping the terms by their exponents $z^i$ we obtain$\colon$
\[
H(\phi_{0}(z),\ldots, \phi_{k}(z))=z^{0}(c_{0}\phi_{0}(0)+\ldots+c_{k}\phi_{k}(0))/0!+\]
\[z^{1}(c_{0}\phi_{0}^{(1)}(0)+\ldots+c_{k}\phi_{k}^{(1)}(0))/1!+\ldots\]
\[
z^{j}(c_{0}\phi_{0}^{(j)}(0)+\ldots+c_{k}\phi_{k}^{(j)}(0))/j!+\ldots
\]
\[
z^{k}(c_{0}\phi_{0}^{(k)}(0)+\ldots+c_{k}\phi_{k}^{(k)}(0))/k!.
\]
The curve $\gamma$ and the hyperplane $\{H=0\}$ intersect with multiplicity $i$ at $p$ if and only if\\ $\bar{H}(z)=H(\phi_{0}(z),\ldots, \phi_{k}(z))=cz^{i}+h.o.t., c\neq0$. We list the first $k+1$ coefficients of the expansion $\bar{H}(z)$ in the following matrix$\colon$
\[A=
\begin{bmatrix}
    \phi_{0}(0)       & \phi_{1}(0) & \phi_{2}(0) & \dots & \phi_{k}(0) \\
    \phi_{0}^{(1)}(0)       & \phi_{1}^{(1)}(0) & \phi_{2}^{(1)}(0) & \dots & \phi_{k}^{(1)}(0) \\
    \hdotsfor{5} \\
    \phi_{0}^{(k)}(0)       & \phi_{1}^{(k)}(0) & \phi_{2}^{(k)}(0) & \dots & \phi_{k}^{(k)}(0)
\end{bmatrix}.
\]
As we can see, the matrix $A$ is indeed the Wronskian of the system of functions $\{\phi_{i}\}$ at zero. Since $\gamma$ does not lie in any hyperplane, the Wronskian matrix is not identically zero. Switching from $p$ to another point $q\in U$ if necessary, we conclude that $A$ has maximal rank. To finish the proof, we choose $V^{i}$ to be the set of hyperplanes corresponding to solutions of the first $i$ rows of $A$.
\end{proof}

By $\mathbb{CP}^{A}$ we denote the projective space whose homogeneous coordinates are enumerated by $A$. We denote by $\mathcal{Z}(\ldots)$ the zero set of an ideal or a set of polynomials. We say that a curve $F$ is supported at $B=(B_1,\ldots, B_{n-1})\subset\mathbb{Z}^n$ if there exists $f\in\mathbb{C}^B$ such that $F=\mathcal{Z}(f)$. The number of points in the set $A$ is denoted by $|A|$. The mixed volume of convex polytopes $A, B_1,\ldots, B_{n-1}$ is denoted by $\mv(A, B_1,\ldots, B_{n-1})$.

The monomial map is denoted by $m_A$ just like in the previous section. Let $g=\sum_{a\in A}c_{a}x^{a}$ be a polynomial supported at $A$. By $\pi_{g}$ we denote the hyperplane in $\mathbb{CP}^{A}$ given by the linear form $\sum_{a\in A} c_{a}z_{a}$, i.e. $\pi_{g}$ has the same coefficients as $g$ has. The next lemma explains the correspondence between hyperplanes in $\mathbb{CP}^{A}$ containing $m_A(F)$ and hypersurfaces in $(\mathbb{C}^*)^n$ which contain $F$.

\begin{lemma}\label{l21}
Let $A, B_1,\ldots, B_{k}$ be a tuple of subsets in $\mathbb{Z}^{n}$. Then there exists an open dense subset $X\subset\mathbb{C}^B$ such that for any $f\in X$ and for any $g\in\mathbb{C}^{A}$ the following conditions are equivalent$\colon$
\begin{enumerate}
    \item $m_{A}(\mathcal{Z}(f))$ is contained in $\pi_{g}$;
    \item $g$ is in the ideal generated by $f$.
\end{enumerate}

If all $B_i$ are the same and $\dim B>k$, there exists an open dense subset $Y\subset\mathbb{C}^B$ such that for all $f\in Y$, $f$ is r-general.  

\end{lemma}

\begin{proof}
($\Rightarrow$)Let $f\in\mathbb{C}^B$. From \cite{Kh} (Theorem 1) it follows that for almost all tuples of coefficients of the system $f=0$, the scheme $\mathbb{C}[x]/(f)$ is smooth. By definition, it means that for any point $q\in\mathcal{Z}(f)$, the corresponding local ring $(\mathbb{C}[x]/(f))_q$ is regular and thus by Auslander–Buchsbaum theorem it is a unique factorisation domain, in particular it is reduced and so is $\mathbb{C}[x]/(f)$. We denote the set of all such $f$ by $X$. Choose arbitrary $x\in\mathcal{Z}(f)$ and denote $m_A(x)$ by $p$. Since $p\in\pi_{g}$ we have $\sum c_{a}x^{a}=0$ therefore $x\in \mathcal{Z}(g)$ and $\mathcal{Z}(f)\subset \mathcal{Z}(g)\Rightarrow$ by Hilbert nullstellensatz for Laurent polynomials $g^r\in (f)$ for some natural $r$. Since $(f)$ is radical, $r$ can be chosen to be 1. 
($\Leftarrow$)For the other implication, let $g=c_1 f_1+\ldots+c_{k}f_{k}$, then the vanishing set of $f$ vanishes $g$ as well hence $\mathcal{Z}(f)\subset \mathcal{Z}(g)$. Therefore $m_{A}(\mathcal{Z}(f))\subset\pi_{g}$.

Finally, from \cite{Kh} (Theorem 17) it follows that for $\dim B>k$, for a sufficiently general system of equations, both tuples $f_1=\ldots=f_{k-1}=0$ and $f_1\vert_F,\ldots, f_{k-1}\vert_F$ for any facet $F$ of $B$ define irreducible varieties which are reduced by the previous paragraph. For a reduced varieties its ideal is equal to the initial ideal defining variety and thus both ideals $(f_1,\ldots,f_{k-1})$ and $(f_1\vert_F,\ldots, f_{k-1}\vert_F)$ are prime for any facet $F$ of $B$. From this it follows that any tuple $f_1,\ldots,f_k$ from $Y$ is regular alongside $(f_1\vert_F,\ldots, f_{k}\vert_F)$.
\end{proof}

From lemma \ref{l21} it immediately follows that $|A|-\mathcal{V}_A^f-1$ is equal to the dimension of the projective hull of $m_{A}(\mathcal{Z}(f))$ for generic $f$ supported at $B$. 
Now, we will prove theorem \ref{th1}
\begin{proof}
Let $F=\mathcal{Z}(f)$ be a smooth curve supported at $B$. Since $\mv(A, B_1,\ldots, B_{n-1})$ is the number of isolated roots of a generic system supported at $(A, B)$ we can see that the number of intersections of $m_A(F)$ and a generic hyperplane in $\mathbb{P}^A$ is positive and hence $m_A\vert_F$ is not degenerate and $\overline{m_A(F)}$ is an algebraic curve. Consider a generic point $p$ of $m_A(F)$ (generic here means non-Weierstrass point and not the image of a branch point). Let the dimension of the projective hull of $m_A(F)$ be $k$. Since $p$ is non-Weierstrass then by Theorem \ref{DGT1}, there exists a tuple of osculating hyperplanes $\{H_{i}\}_{i=0}^{k}$ such that $H_{i}$ intersects $m_A(F)$ at $m_{A}(p)$ with multiplicity $i$. We denote the coefficients of these hyperplanes by $H_{a}^{i}$. Since $F$ is a smooth curve and $p$ is not the image of a branch point of $m_A\vert_F$, we can choose a holomorphic parametrization in some neighbourhood of $p\colon z_i=z_i(t)$. For each $i\in\{0,\ldots,k\}$ define $g_{i}=\sum_{a}H_{a}^{i}z^a\in\mathbb{C}^{A}$, $G_{i}=\mathcal{Z}(g_{i})$. Substitute $z_i(t)$ into the equation of $g$, and write the Laurent series for $g$
$$
g=g_{m}t^m+h.o.t.
$$
where $m$ is the intersection multiplicity of $F$ and $G_{i}$ at $p$. Let $L$ be a linear form not vanishing at $m_{A}(p)$. Then $\sum_{a}H_{a}^{i}z_a/L$ is a meromorphic function that has the same order at $m_{A}(p)$ as $g$ has at $p$. Indeed, substituting $z_i(t)$ into the function $\sum_{a}H_{a}^{i}z_a/L$ we will obtain $(g_{m}t^m+h.o.t.)/(c_{0}+c_{1}t^{l}+h.o.t.), c_{0}, c_{1}\neq0, l>0$. Thus we have the Laurent expansion of the following form
$$
(g_{m}t^m+h.o.t.)/(c_{0}+c_{1}t^{l}+h.o.t.)=c^{-1}_{0}(g_{n}t^n+h.o.t.)(1-c_{1}c^{-1}_{0}t^{l}+h.o.t.)=c^{-1}_{0}g_{m}t^m+h.o.t.
$$
Therefore, the intersection multiplicities of $F$ and $G_{i}$ at $p$ and $m_A(F)$ and $H^{i}$ at $m_{A}(p)$ are both equal to $i$ as desired.
\end{proof}

\section{\bf Strata of the dual variety of a projective curve.}\label{str}

In this section we describe the structure of the incidence space of curves and osculating hyper-surfaces with given support, that were introduced in the previous sections. We show that the natural stratification of this space corresponds to the stratification of the discriminant variety in one of its irreducible components. We work with the partial case when all $B_i$ are the same, but we sometimes enumerate them like they are different to avoid confusion in the notation and make formulas involving dimensions and summations more readable.

To be precise, we formulate that as a theorem$\colon$

\begin{theor}\label{STR}
Let $(A, B)=(A, B_1, \ldots, B_{n-1})$ be a tuple of subsets from $\mathbb{Z}^n$ such that all $B_i$ are the same, $\mv(A, B)>1$ and $\dim B=n$. Then the discriminantal variety $\triangle_{(A, B)}$ possess a chain of non empty consecutive strata

$$
S_{|A|-\mathcal{V}^B_A-1}\subset\ldots\subset S_i\subset\ldots\subset S_{2}\subset\triangle_{(A, B)}
$$

where $S_i$ contains a dense subset consisting of systems with root of multiplicity $i$ and $\dim S_i=\sum|B_i|+|A|+1-i=n|B|+|A|+1-i$.
 \end{theor}
Let $\Gamma$ be a projective curve that is given locally (in some neighbourhood $U$ of a generic point $p\in\Gamma$) by a family of analytic functions$\colon$ 
\[
\phi\colon U\rightarrow\mathbb{P}^n
\]
\[
z\mapsto[\ldots\colon \phi_{i}(z)\colon\ldots].
\]
Let the dimension of the projective span of $\Gamma$ equal $k$. It is well known that any set of analytic functions is linearly independent if and only if the Wronskian of this set is not identically 0. Recall that the Wronskian can be computed by the following formula
\[\det
\begin{bmatrix}
    \phi_{0}       & \phi_{1} & \phi_{2} & \dots & \phi_{k} \\
    \phi_{0}^{(1)}       & \phi_{1}^{(1)} & \phi_{2}^{(1)} & \dots & \phi_{k}^{(1)} \\
    \hdotsfor{5} \\
    \phi_{0}^{(k)}       & \phi_{1}^{(k)} & \phi_{2}^{(k)} & \dots & \phi_{k}^{(k)}
\end{bmatrix}.
\] With the Wronskian matrix, it can be easily shown that the coefficients of the osculating hyperplanes of tangency $i=0,\ldots,k$ depends holomorphicaly on $z$ in some neighbourhood of $p$, e.g. \cite{MRS}. If $p=p(\omega), \omega\in U$ is a point on the curve viewed as a holomorphic function of one complex variable and $\phi$ is some fixed embedding we denote by $W_\phi(\omega)$ the determinant of the Wronskii matrix as a function of $\omega$. We would like to apply all this machinery to the curve supported at $B$ embedded into projective space via $m_A$. 

Let $I=(f_1, \ldots, f_{n-1})\subset\mathbb{C}[\ldots,z^{\pm}_i,\ldots], f_i\in\mathbb{C}^{B_i}$ be an ideal generated by polynomials supported at $B$ and $g$ be a polynomial supported at $A$. Assume that $g\in I$. We know that for a given $f$ the set of all such $g$ is a finite-dimensional vector space over $\mathbb{C}$. We denote this vector space by $\mathcal{V}_{A}^{f}$ according to the agreement established in the introductory section. For generic $f$, $\dim\mathcal{V}^{f}_{A}$ does not depend on the choice of $f$. We will show that in proposition \ref{Generic}.
So, we ignore $f$ in the expression $\mathcal{V}^{f}_{A}$ and use $\mathcal{V}^{B}_{A}$ instead. Let $\Gamma=\overline{m_A(\mathcal{Z}(f))}$.

There exists an algebraic subset $\triangle_B\subset\mathbb{C}^B$ such that for any curve $\{f=0\}$ in the complement of $\triangle_B$ we have $\dim\mathcal{V}_A^{f}=\dim\mathcal{V}_A^{B}$. 

Let $p\in\Gamma$. We will call $p$ {\it flag generic} if it is smooth and non-Weierstrass. Consider the incidence sets 
$$
\Sigma=\{(p, f, g)\vert\ m_{A}(p)\in \Gamma\ \text{is flag generic},\ f\notin\triangle_B,\ T_p(\Gamma)\subset\mathcal{Z}(H_g)\}\subset(\mathbb{C}^{*})^n\times\mathbb{C}^{B}\times\mathbb{C}^{A}
$$
$$
\Sigma_0=\{(p, f)\vert\ m_{A}(p)\in \Gamma\ \text{is flag generic},\ f\notin\triangle_B\}\subset(\mathbb{C}^{*})^n\times\mathbb{C}^{B}
$$

We have natural projections 
\begin{center}

\begin{tikzcd}
\Sigma \arrow[rd, "\pi"] \arrow[r, "D"] & \mathbb{C}^{B}\times\mathbb{C}^{A} \\
& \Sigma_0 \arrow[r, "\pi_0"]  & \mathbb{C}^B\setminus\triangle_B\\

\end{tikzcd}
 
\end{center}

by definition $\forall f\in\mathbb{C}^B\setminus\triangle_B$, $\overline{\pi_0^{-1}(f)}=\{(p, f)\vert\  m_{A}(p)\in\overline{m_{A}(\mathcal{Z}(f))}\}$ is equal to $\overline{m_{A}(\mathcal{Z}(f))}$ and thus irreducible of constant dimension 1. Taking the closure of $\Sigma_0$ we have
$$
\overline{\pi_0}\colon\overline\Sigma_0\rightarrow\mathbb{C}^B
$$
where $\overline{\pi_0}$ is just the projection on the second factor and it coincides with $\pi_0$ on $\Sigma_0$.

\begin{lemma}
$\Sigma_0$ is an irreducible subset of $(\mathbb{C}^{*})^n\times\mathbb{C}^B$.
\end{lemma}

\begin{proof}
As we can see $\overline{\Sigma_0}$ is given by the system of equations
$$
f_1=f_2=\ldots=f_{n-1}=0\ \text{in}\ (\mathbb{C}^{*})^n\times\mathbb{C}^B
$$
For each $i\in\{0,\ldots, n-1\}$ take any $b(i)\in B_i$ and consider the equivalent system 
$$
\tilde{f}_i=f^{b(i)}_i+z^{-b(i)}\overline{f}_i=0
$$
where $\overline{f}_i=\sum_{b\in(B_i\setminus b(i))}f^{b}_i z^{b}$. Ideals that are generated by both systems of equations are equivalent. Consider the quotient ring 

$$
\mathbb{C}[\overline{\Sigma}_0]=\mathbb{C}[\ldots, z_i^{\pm},\ldots][\ldots,f_i^{b_i^j},\ldots]/(f_1,\ldots, f_{n-1})
$$
where $b_i^j\in B_i$, for given $i$, $j=0,\ldots, |B_i|$. Since $z^{-b_i}\overline{f}_i\in\mathbb{C}[\ldots, z_i^{\pm},\ldots][\ldots,f_i^{b_i^j},\ldots]$, we have an isomorphism
$$\mathbb{C}[\ldots, z_i^{\pm},\ldots][\ldots,f_i^{b_i^j},\ldots]/(f_1,\ldots, f_{n-1})\backsimeq\mathbb{C}[\ldots, z_i^{\pm},\ldots][\ldots,f_i^{b_i^j},\ldots]/(f^r_1,\ldots, f^r_{n-1})=$$
$$
=\mathbb{C}[\overline{\Sigma}_0]=\mathbb{C}[\ldots, z_i^{\pm},\ldots][\ldots,f_i^{b_i^j},\ldots],\ \text{and in each variable}\ f_i^{b_i^j}\ \text{we have}\ j\neq b(i).
$$

We deduce that the ideal $(f_1,\ldots, f_{n-1})$ is prime. Then $\overline{\Sigma}_0$ is irreducible. Since $\Sigma_0$ is a dense subset of $\overline{\Sigma}_0$ it is also irreducible.
\end{proof}

Let us compute the dimension of $\overline{\Sigma}_0$. For all $f\in\mathbb{C}^B\setminus\triangle_B$ we have $\dim\overline{\pi}_0^{-1}(f)=1$, so $\dim\Sigma_0=\dim\overline{\Sigma}_0=\dim\mathbb{C}^B+1=\sum |B_i|+1=n|B|+1$.

Projection $\pi\colon\Sigma\rightarrow\Sigma_0$ turns $\Sigma$ into a bundle over $\Sigma_0$ whose fibers are linear spaces of osculating hypersurfaces supported at $A$. If $\mv(A, B_1,\ldots, B_{n-1})>1$, $\Gamma$ is not a line. Osculating hypersurfaces are in correspondence with osculating hyperplanes in $\mathbb{CP}^{|A|-1}$ so the dimension of each fiber is $\dim\mathbb{CP}^{|A|-1}-1=|A|-2$ since any curve with non-zero curvature is dual effective. We conclude that $\Sigma$ is irreducible of dimension $\sum |B_i|+1+|A|-2=\sum |B_i|+|A|-1$ as a total space of vector bundle over irreducible base.

By the definition of $\Sigma$, $D(\Sigma)$ and thus $D(\overline{\Sigma}_0)$ is an irreducible subset of the discriminant of maximal dimension. 
Consider the subbundle

$$\Sigma^{i}=\{(p, f, g)\vert\ m_{A}(p)\in \Gamma\ \text{is flag generic},\ f\notin\triangle_B,\ \mathcal{Z}(H_g)\in V_p^{i}\}\subset(\mathbb{C}^{*})^n\times\mathbb{C}^{B}\times\mathbb{C}^{A}\subset\Sigma$$

where $V_p^{i}$ was defined above in \ref{DGT1}. Let us prove that $\cod\Sigma^i-\cod\Sigma=i-2$.

\begin{theor}\label{th3}

\indent

In the above notation $\pi\colon\Sigma^i\rightarrow\Sigma_0$ is a vector bundle $\forall\ i=2,\ldots, |A|-\dim\mathcal{V}^A_B-1$. Its total space has dimension $\dim\Sigma_0+|A|-i=\sum|B_i|+|A|-i+1=n|B|+|A|-i+1$. 
\end{theor}

\begin{proof}
It is sufficient to find a trivialising neighbourhood for each $(p, f)\in\Sigma_0$. Let $(p, f)\in\Sigma_0$. Consider the Jacobian matrix of the system $f=0$ where $f_i$ is considered as a polynomial in variables $z_i$ and $f_i^{b_i^j}$. It has the following form
$$
\begin{bmatrix}
J\vert\ldots
\end{bmatrix}.
$$
where $J$ is the Jacobian matrix of $f$ with respect to the variables $z_i$. That means that $\Sigma_0$ lies in the smooth locus of $\overline{\Sigma}_0$. Pick a holomorphic parametrization of the point $(p, f)\in\Sigma_0$ in some open neighbourhood of $U\subset\mathbb{C}^{\sum|B_i|+1}$. Consider the embedding of the curve $f=0$ into $\mathbb{CP}^{|A|-1}$ via the map $m_A$. Denote $p=p(\omega), \omega\in U$ as a holomorphic parametrization of $p$ in $\Sigma_0$. Under assumption $m_{A}(p)\in \Gamma$ is smooth and non Weierstrass. That means $W_{m_A}(\omega)$ is not vanishing in some neighbourhood of $(p, f)$, probably smaller than $U$. Choose the first $i$ rows of the Wronskii matrix as a matrix of a system of linear equations. This system has maximal rank and thus its set of solution depends on $|A|-i$ free parameters $(c_{i+1},\ldots, c_{|A|})$. Now the desired trivialisation is given by the map 
$$
s_U\colon\pi^{-1}(U)\rightarrow U\times\mathbb{C}^{|A|-i}
$$
$$
(p, f, g)\mapsto(\omega, c_{i+1},\ldots, c_{|A|})
$$
\end{proof}

This theorem shows that we have a length $k$ descending chain of strata of consecutive dimensions in $\Sigma$, such that every system in the dense subset of $i$-th stratum has a root of multiplicity exactly $i$. We will use it now to conclude the section.

\begin{proof}(Of theorem \ref{STR})
In the construction above, put $S^0_i=\Sigma^i$. By the definition of $\Sigma^i$ and theorem \ref{th3}, all $S^0_i$ have expected consecutive dimensions and are subsets of $\Sigma$. Since $D\colon\Sigma\rightarrow\mathbb{C}^{B}\times\mathbb{C}^{A}$ has a finite set of points in each fibre, $\dim\Sigma^i=D(\Sigma^i)$. To finish the proof, we define $S_i$ to be $D(\Sigma^i)$.
\end{proof}

\section{\bf Idea behind the proof of the main result.}\label{mex}
    Let us denote the standard 2-simplex by $E$. Let $f_1, f_2$ be two generic polynomials of degree $d_1, d_2$ respectively. Let us find $\dim\mathcal{V}_A^{f}$, where $f=(f_1, f_2)$ and $A=E$.
    Let $c=(c_1, c_2)$ where $c_1, c_2$ are polynomials of degree $D_1, D_2$ respectively, so $(c_1, c_2)\in\mathbb{C}^{D_1 E}\oplus\mathbb{C}^{D_2 E}$. Our goal is to understand for what values of $D_1, D_2$ the corresponding pair $C=(D_1 E, D_2 E)$ satisfies $\mathcal{V}_A^{f}=\mathcal{V}_A^{C,\ f}$.

    We consider the following map$\colon$

    $$
    \phi\colon\mathbb{C}^{D_1 E}\oplus\mathbb{C}^{D_2 E}\rightarrow\mathbb{C}[x, y]
    $$
    $$
    (c_1, c_2)\mapsto c_1 f_1+c_2 f_2
    $$

On the level of linear algebra we are looking for $(c_1, c_2)$ such that $c_1 f_1+c_2 f_2$ is nonzero and has no monomials of degree higher than $1$. For arbitrary linear map $\psi\colon V\rightarrow U$ and $W\subset V$ a linear subspace of $V$ given by a system of linear equations with matrix $\Omega$, we have the following simple formula$\colon$

$$
\dim\im\psi\vert_W =\dim V-\rk\Omega-\dim(W\cap\ker\psi)
$$

To apply this formula to our problem we put $W=\{(c_1, c_2)\vert \deg(c_1 f_1+c_2 f_2)\leq1\}$. And $\psi=\phi$ defined above. In our case $\ker\phi\subset W$ since a zero polynomial has degree lower than 1. So

\begin{equation}\label{eq2}
    \dim\mathcal{V}_A^{C, f}=\dim V-\rk\Omega-\dim\ker\phi
\end{equation}

First, we are ready to calculate $\ker\phi$. This is the space of polynomials $(c_1, c_2)$ such that $c_1 f_1+c_2 f_2=0$. We use the fact that $f_1, f_2$ are generic and thus co-prime elements of the polynomial ring. That means we can put $c_1=\lambda f_2$, $c_2=\mu f_1$ and substituting it to the equation we obtain $(c_1, c_2)=\lambda(f_2, -f_1)$. The last question is what we can say about $\lambda$? We can easily show using proposition \ref{prop1} that $\lambda$ is any polynomial in the space $\mathbb{C}^{D_1 E\ominus d_2 E}$ and simultaneously any polynomial in $\mathbb{C}^{D_2 E\ominus d_1 E}$.
It is easy to see that $D E\ominus d E=(D-d)E$. So, we see that $\lambda$ is an arbitrary polynomial of degree $\leq\min(D_1-d_2, D_2-d_1)$. 

Next, let us calculate $\dim W$. Subspace $W$ is given by the system of linear equations. Each one corresponds to a monomial in $\supp c_i f_i$ that lies outside of $A$. We denote the matrix of conditions by $\Omega$ just like in the formula above. For convenience we enumerate rows of $\Omega$ by vanishing monomials and columns by coefficients in $(c_1, c_2)$. For example, if $g=(1+x)(c_0+c_1 x+c_2 x^2)$ and $g$ is of degree 1 then $\Omega$ has the following form$\colon$

$$
\Omega=\begin{bmatrix}
 0   &  1 & 1\\
0  &  0 & 1
\end{bmatrix}.
$$

In practice it is really difficult to get any information from $\Omega$. We can partially avoid the complexity by considering variation of $\dim\mathcal{V}_A^{C,\ f}$ instead of using the direct approach. By \ref{eq2} we have$\colon$

$$
\delta\dim\mathcal{V}_A^{C,f}=\delta\dim V-\delta\rk\Omega-\delta\ker\phi
$$
where $\delta\dim\mathcal{V}_A^{C,f}=\delta\dim\mathcal{V}_A^{C_2,f}-\delta\dim\mathcal{V}_A^{C_1,f}$, $C_2=((D_1+1)E,(D_2+1)E),\ C_1=(D_1 E, D_2 E)$.

As we can see 
\begin{equation}\label{eq3}
\Omega_2=\left[\begin{array}{c|c c} 
	\Omega_1 & *\\ 
	\hline 
	0 & \delta\Omega
\end{array}\right]
\end{equation}

Where $\Omega_1$ is the matrix of conditions corresponding to $C_1=(D_1 E, D_2 E)$. We obtain new columns because $C_2$ contains an additional face and for the same reason we obtain additional rows. We have zero block under $\Omega_1$ because obviously the previous set of coefficients is not involved into new conditions obtained from new points in $C_2$. We can see from figure \ref{exf3} that additional degree leads to additional degree in the product $f_i c_i$ so all new conditions correspond to the new face.

\begin{figure}[h!]
\includegraphics[width=0.5\linewidth]{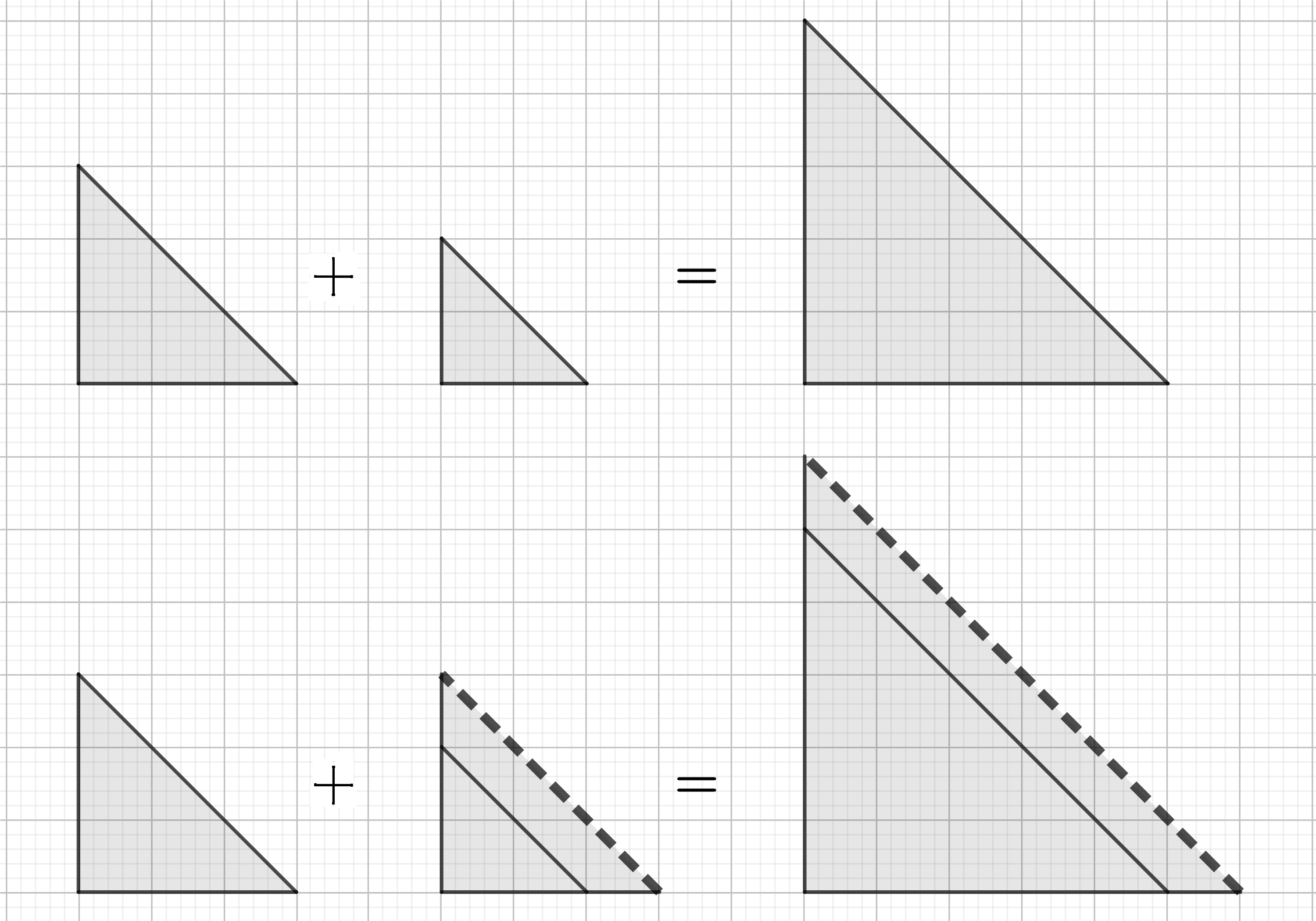}
\caption{new face in $C$ and in $C+B$}
\label{exf3}
\end{figure}
In $\Omega_2$ we can consider $\delta\Omega$ in more detail. Note that it is located in the intersection of new rows and columns. Since all new conditions are located in the common face we can group up all linear equations in $\delta\Omega$ and obtain single algebraic equation$\colon$

$$\delta c_1 f^\delta_1+\delta c_2 f^\delta_2=0$$

Where $f_i^\delta=f_i\vert_{d_i E(-1,-1)}$. We deduce that $(\delta c_1, \delta c_2)=\mu(f_2^\delta, -f_1^\delta)$, $\mu\in\mathbb{C}^{D_1 E(-1,-1)\ominus d_2 E(-1,-1)}\cap\mathbb{C}^{D_2 E(-1,-1)\ominus d_1 E(-1,-1)}$.
If $D_1+d_1>D_2+d_2$ then coefficients of $\delta c_1$ will automatically vanish. So without loss of generality we assume that $D_1+d_1=D_2+d_2$. And thus $\mu\in\mathbb{C}^{D_1-d_2+1}=\mathbb{C}^{D_2-d_1+1}$. 

For matrices of the form \ref{eq3} we have $\rk\Omega_2\geq\rk\Omega_1+\rk\delta\Omega$. Thus, $\delta\rk\Omega=\rk\Omega_2-\rk\Omega_1\geq\rk\delta\Omega$ and $-\delta\rk\Omega\leq-\rk\delta\Omega$. We obtain, by the dimension theorem for vector spaces $\delta\dim V-\rk\delta\Omega=\dim\mathbb{C}^{D_2-d_1+1}$. We have computed all terms we needed. Finally, we substitute all data into the variation of $\delta\dim\mathcal{V}_A^{C, f}$

$$
\delta\dim\mathcal{V}_A^{C, f}=\delta\dim V-\delta\rk\Omega-\delta\ker\phi\leq
$$

$$
\leq\delta\dim V-\rk\delta\Omega-\delta\ker\phi=
$$

$$
=D_1-d_2+2-(D_1-d_2+2)=0.
$$
We can see that by increasing degrees of $C$ we make no affect on $\dim\mathcal{V}_A^{C, f}$. Without loss of generality assume that $d_1\leq d_2$. The minimal $(D_1, D_2)$ satisfying the property $D_1+d_1=D_2+d_2$ is $(d_2-d_1,0)$. We will take it from here.

\begin{enumerate}
    \item $d_2>d_1\geq2$. 

Let us find all linear $l$ that can be represented in the form $l=c_1 f_1+c_2 f_2$, $c_2\in\mathbb{C}^*$. If we restrict the system to the face $E(-1,-1)$, as before, we obtain that $f_2\vert_{d_2 E(-1,-1)}$ is divisible by $f_1\vert_{d_1 E(-1,-1)}$. And this leads to a contradiction. So, in this case $\dim\mathcal{V}_A^{f}=0$.
    
    \item $d_2>d_1=1$. In a similar way we obtain $l=c f_1$, $c\in\mathbb{C}$. So $\dim\mathcal{V}_A^f=1$
    \item We easily obtain $\dim\mathcal{V}_A^f=2$.
\end{enumerate}

In section \ref{ga3} we will generalize this method to arbitrary dimension $n$ for $2$ bodies in the ring of Laurent polynomials.

\section{\bf Existence and properties of normal chains of a convex polytope.}\label{NC}

\subsection{Chains of positive type.}

In the previous section we demonstrated technique of `moving facets' of $C$ to find the variation of $\dim\mathcal{V}_A^{C, f}$ and finally conclude that it is zero. The main reason we could do this is that $C$ and $B$ had the same shape and any element of $C_{(-1,-1)}\ominus B_{(-1,-1)}$ can be extended to the element of $C\ominus B$. We would like to generalize this process to convex polytopes which are more complex than the unit simplex.  

Let $X\subset\mathbb{R}^n$ be a convex polytope of maximal dimension. Let $F$ be the number of its facets. We can represent $X$ as a set of solutions of a system of $F$ inequalities $x.\alpha_i\leq \beta_i$, $i\in\{1,\ldots, F\}$. If $h_O^\epsilon\colon\mathbb{R}^n\rightarrow\mathbb{R}^n$ is a fixed homothety with the center in $O$ and factor $\epsilon$, we will denote $h_O^\epsilon(X)$ by $\epsilon X$ for short. In the case when $O$ is the origin and it lies in the interior of $X$, $\epsilon X$ can be represented in the form $x.\alpha_i\leq\epsilon\beta_i$, $i\in\{1,\ldots, F\}$. Everywhere we assume that the center of homothety is always in the origin since it is more convenient for notation and computations.  

\begin{defin}\label{def1}
For a given convex polytope $X$ given by the set of inequalities $\{x.\alpha_i\leq \beta_i\}_i=1^F$ we will say that the chain of convex polytopes 

$$
 X\subset\epsilon_0 X\subset X_1\subset X_2\subset\ldots\subset X_{jL+i}\subset\ldots
$$
    is a positive normal chain if we have the following$\colon$

\begin{enumerate}
    \item for $\epsilon_0>1$ we have $\mathbb{Z}(\epsilon_0X)=\mathbb{Z}(X)$.
    \item for all $j\in\mathbb{N}$ we have $X_{jF}=(\epsilon_0+js)X$ for some $s>0$. 
    \item for all $i\in\{1,\ldots, F-1\}$ and for all $k\in\mathbb{N}$, $X_{jF+i}$ is given by the system $$\{x.\alpha_1\leq (\epsilon_0+js+s)\beta_1, \ldots,\ x.\alpha_i\leq (\epsilon_0+js+s)\beta_i,\ x.\alpha_{i+1}\leq (\epsilon_0+js)\beta_{i+1},$$ $$\ldots,\ x.\alpha_{F-1}\leq (\epsilon_0+js)\beta_{F-1}\}$$  
    \item $\cup_1^{\infty} X_k=\mathbb{R}^n$.
    \item for any $k\in\mathbb{N}$ from the chain we have $X_{k}=X+(X_k\ominus X)$. 
    \item For any $k\in\mathbb{N}$ we have $\mathbb{Z}(X_{k+1}\setminus X_{k})$ is either empty or lies in a hyperplane.
\end{enumerate}

We denote positive normal chains $X\subset X^{+}_*$ for short.

\end{defin}

We would like to give some clarification. Except the second one, each term $X_k$ in the chain is derived from the previous one $X_{k-1}$ by shifting one of its facets. We require that such deformations do not change the set of faces of $X_{k-1}$, and when all faces are used, we end up with a pure homothety of $X$. 

The most interesting part of this definition is how we keep the property (5) satisfied along the normal chain. Once it is achieved all other properties will be satisfied automatically. The reality is that we don't have to pick $X$ as a starting point of a normal chain, any convex polytope $Y$ with the property $Y=X+(Y\ominus X)$ fits.

\begin{defin}
    Let $Y$ be given by the system of linear inequalities $\{\alpha_i.x\leq b_i\}_{i=1}^{N}$ where $N$ is the number of facets of $Y$. We will denote $\{\alpha_i.x\leq tb_i(1+s_i)\}_{i=1}^{F}$ by $X(t,s)$ for any $s=(s_1,\ldots,s_i,\ldots, s_F)$, where $F$ is the number of facets of $X$. We will say that $Y$ is stably $X$ decomposable if there exists $\epsilon>0$ such that $Y(1,s)=X+(Y(1,s)\ominus X)$ for all $s_i$ from the range $0<s_i<\epsilon$.
\end{defin}

It is not very difficult to give an example of a couple $X, Y$ such that $Y$ is $X$ decomposable but not stably decomposable \ref{fig2}.

\begin{exa}
Here $X=Y$ and $s=(0,0,0,\delta)$ for some small $\delta$, since we have $4$ facets and we move only one of them.

\begin{figure}[h!]
  \includegraphics[width=0.7\linewidth]{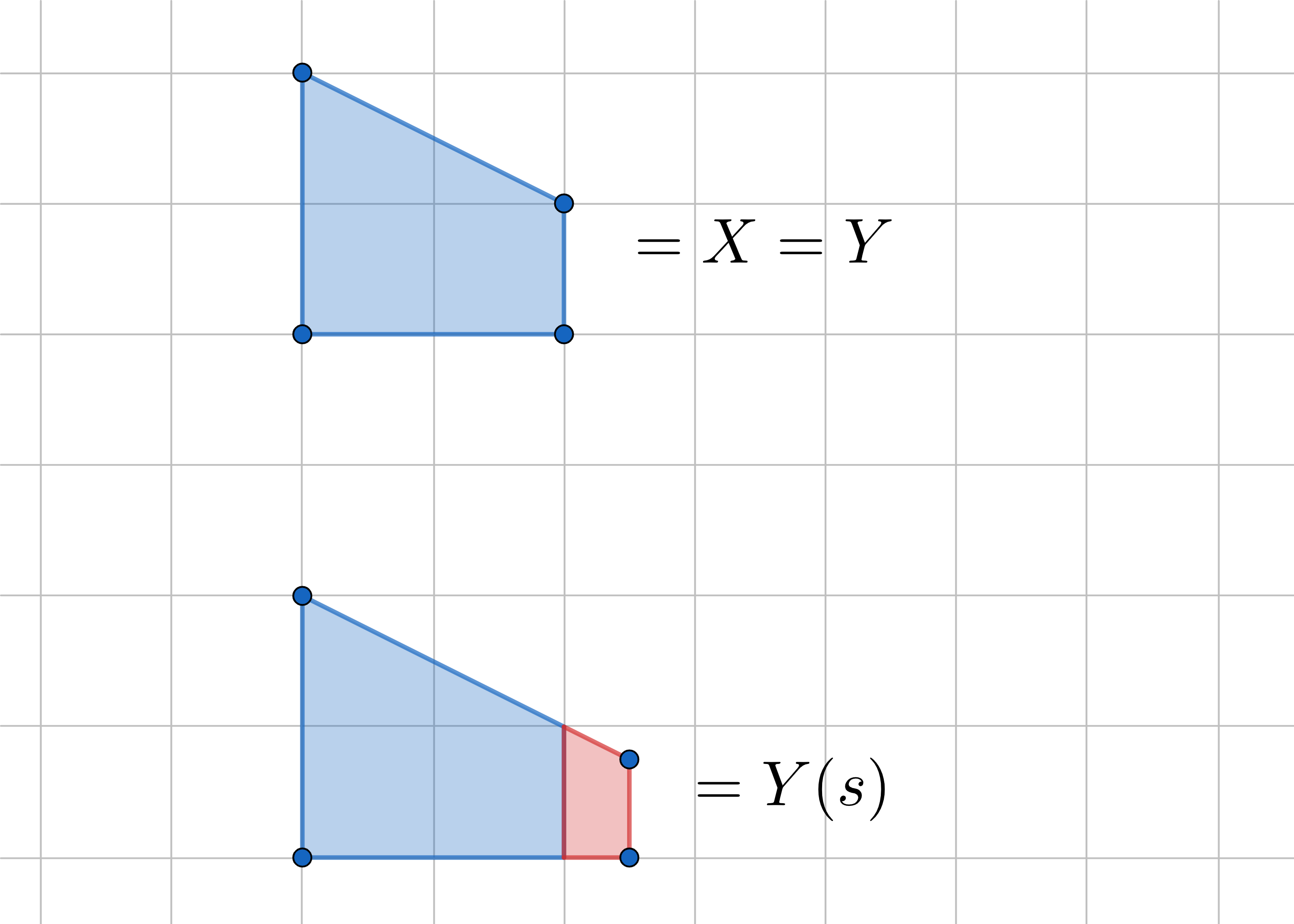}
  \caption{non-stable $Y$}\label{fig2}
\end{figure}    
\end{exa}

What we are looking for is $\colon$

\begin{theor}\label{TH}
    Let $X, Y\subset\mathbb{R}^n$ and $Y$ is $X$ decomposable, then $tY$ is stably $X$ decomposable for any $t>1$. Moreover, if $Y(t,s)$ is $X$ decomposable then $Y(T,s)$ is also $X$ - decomposable for any $T>t$.
\end{theor}

We will prove this result by using the so-called support functions. The main geometric idea we will use is convexity. If for each vertex of $Y$ we could find a decomposition $v=u+y$ for some $x\in X$ and $y\in Y\ominus X$, then any other point of $Y$ can be represented in a similar way because $\alpha v_1+(1-\alpha)v_2=(\alpha x_1+(1-\alpha)x_2)+(\alpha y_1+(1-\alpha)y_2)$ and any point in $Y$ can be represented as a convex combination of some vertices.

Now, let $X,Y\subset\mathbb{R}^n$ such that $Y$ is $X$ decomposable. We want to show that $tY$ is stably $X$ decomposable for $t>1$. We start from the description of $A\ominus B$ for a polytope $A$ and more or less arbitrary $B$.

\begin{defin}
    Let $B$ be a subset of $\mathbb{R}^n$. For a given $\alpha\in\mathbb{R}^n$ we define the support function$\colon$
$$
h_B(\alpha)=\sup_{x\in B}\alpha.x
$$
    
\end{defin}

If for each facet $F$ of $A$ with the corresponding normal vector $\alpha$ we can define the support function $h_B(\alpha)$ then it turns out that $A\ominus B$ is also a polytope.

\begin{theor}(T2.3 from \cite{G})\label{ominus}
    Let $A$ be defined by the system of inequalities $\colon$

$$
A=\{\alpha_i.x\leq\beta_i\}
$$

For simplicity we assume that $i$ enumerates the set of facets. For each $i$ let $h_B(\alpha_i)$ be well defined. Then 

$$
A\ominus B=\{\alpha_i.x\leq\beta_i-h_B(\alpha_i)\}
$$

\end{theor}

We are two propositions away from proving theorem \ref{TH}.

\begin{utver}\label{P_2}
    If $Y$ is $X$ decomposable then for any vertex $v$ of $Y$ there exists a vertex $u$ of $X$ such that $v=u+y$ for some $y\in Y\ominus X$. 
\end{utver}

\begin{proof}
    Since $Y=X+(Y\ominus X)$ for any vertex $v$ of $Y$ we can provide a decomposition $v=u+y$ where $u\in X$ and $y\in Y\ominus X$. Let $u$ be an interior point of some face of $X$. Then there exist $a,b\in X$ such that $u\in\inter[a,b]\subset X$. Since $[a,b]\subset X$ and $y\in Y\ominus X$, we have $[a,b]+y\subset Y$. But $v=x+y$ is the interior point of the segment $[a,b]+y$ which is impossible because $v$ is a vertex of $Y$.
\end{proof}

and

\begin{utver}\label{P_3}
    If $Y$ is $X$ decomposable then $tY$ is $X$ decomposable.
\end{utver}

\begin{proof}
    We represent $Y$ as the set of solution of the system of inequalities $Y=\{\alpha_i.x\leq\beta_i\}_{i=1}^F$. For any vertex $v$ of $Y$ there exists a vertex $u$ of $X$ such that $v=u+y$ for some $y\in Y\ominus X$ by proposition \ref{P_2}. We want to show that $tv-u\in tY\ominus X$. By \ref{ominus} we know that $tY\ominus X=\{\alpha_i.x\leq t\beta_i-h_X(\alpha_i)\}_{i=1}^{F}$. Since $Y=X+(Y\ominus X)$ we know that $\alpha_i.(v-u)\leq\beta_i-h_X(\alpha_i)$. Since $v$ is a vertex of $Y$ we know that $\alpha_i.v\leq\beta_i\Rightarrow(t-1)\alpha_i.v\leq(t-1)\beta_i$ for $t>1$. From

    $$
    \alpha_i.(v-u)\leq\beta_i-h_X(\alpha_i)
    $$
    $$
    \alpha_i.(t-1)v\leq(t-1)\beta_i
    $$
    We deduce $\alpha_i(tv-u)\leq t\beta_i-h_X(\alpha_i)$ by side-wise addition. That is exactly what we wanted 
 \end{proof}

Finally, we are ready to prove the main result of the course project.

\begin{proof}(of  theorem  \ref{TH})
    Let $Y$ be given by the system of inequalities $Y=\{\alpha_i.x\leq\beta_i\}^{F}_{i=1}$ where $F$ is the number of facets of $Y$. Since $Y$ is not necessarily simple, some vertices of $Y$ meet more than $n$ facets. Let $v$ be one of the vertices of $Y$ and $v(t,s)$ be one of the vertices appeared after the splitting of $tv$. So that $v(t,0)=tv$. By proposition \ref{P_2} we can find a vertex $u$ of $X$ and $y\in Y\ominus X$ so that $v=u+y$. We want to show that $v(t,s)-u\in Y(t,s)\ominus X$. We remind the reader that $Y(t,s)$ can be represented in the form $\{\alpha_i.x\leq tb_i(1+s_i)\}_{i=1}^{F}$. We consider two possible cases$\colon$

\begin{enumerate}
    \item $v$ lies on the hyperplane $\alpha_i.x=\beta_i$. In this case $\alpha_i.v=\beta_i$ and since $v-u\in Y\ominus X$ we have $\alpha_i.(v-u)\leq\beta_i-h_X(\alpha_i)\Rightarrow-\alpha_i.u\leq-h_X(\alpha_i)$. From

$$
-\alpha_i.u\leq-h_X(\alpha_i)
$$
$$
\alpha_i.v(t,s)\leq t\beta_i(1+s_i)
$$

We deduce $\alpha_i.(v(t,s)-u)\leq t\beta_i(1+s_i)-h_X(s_i)$. Thus $v(t,s)-v\in Y(t,s)\ominus X$.

    \item $v$ does not lie on the hyperplane $\alpha_i.x=\beta_i$, so actually $\alpha_i.v<\beta_i$ because $v$ is still a vertex of $Y$. From 
    $$
    \alpha_i.(t-1)v<(t-1)\beta_i
    $$
    $$
    \alpha_i.(v-u)\leq\beta_i-h_X(\alpha_i)
    $$
 We deduce $\alpha_i.(tv-u)<t\beta_i-h_X(\alpha_i)$. Now, consider 

$$
L(t,s)=\alpha_i.(v(t,s)-u)
$$
$$
R(t,s)=t\beta_i(1+s_i)-h_X(\alpha_i)
$$
 Since both $L(t,s), R(t,s)$ are continuous in some neighborhood of zero and $L(0)<R(0)$, we deduce that for sufficiently small $\epsilon>0,\ L(t,s)< R(t,s)$ for all $s$ such that $|s|<\epsilon$. Moreover, we can see that $R(t,0)-L(t,0)=t\beta_i-h_X(\alpha_i)-\alpha_i(tv-u)=t(\beta_i-\alpha_i.v)-h_{X}(\alpha_i)+\alpha_i.u$, which means if $2\epsilon<R(t,0)-L(t,0)$ then $2\epsilon<R(T,0)-L(T,0)$ for any $T>t$ which finishes the second part of the proof.
\end{enumerate}
    We have shown that $v(t,s)-u\in Y(t,s)\ominus X$ for sufficiently small $s$.
\end{proof}

We need to start a normal chain of $X$ from a homothety which includes the same set of integral points as $X$ does as the definition requires.

\begin{lemma}\label{Lemma1}
    Let $X$ be an arbitrary polytope with integral vertices. Then there exists $\epsilon_0>1$ such that $\mathbb{Z}(\epsilon_0 X)=X$.
\end{lemma}
\begin{proof}
    Consider the function 
    
    $$N\colon[1,\infty)\rightarrow\mathbb{N}$$

    $$\epsilon\mapsto\vert\epsilon X\cap\mathbb{Z}^n\vert$$
It is clear that $t_2 X\subset t_1 X$ for $1<t_2<t_1$ thus $N(t_1)\geq N(t_2)$. We have two possible options
\begin{enumerate}
    \item for all $\epsilon$ in some neighbourhood $U$ of $1$ we have $N(\epsilon)=\vert X\cap\mathbb{Z}^n\vert$
    \item $N(1)=\vert X\cap\mathbb{Z}^n\vert$, but for any $\epsilon>1$ we have $N(\epsilon)>\vert X\cap\mathbb{Z}^n\vert$.
\end{enumerate}

In the first case we can take any $\epsilon\in U\setminus 1$ and finish the proof. In the second case for any $t>1$ there exists a point $p(t)\in\mathbb{Z}^n$ such that $p(t)\in tX\setminus X$. We obtain a contradiction since there are only finite number of integral points in $t X\setminus X$ for any finite $t$.
\end{proof}

Let us take $\epsilon_0$ like in the theorem above. Since $X$ is $X$ - decomposable, we have that $\epsilon_0 X$ is stably $X$ - decomposable, hence there exists $\epsilon>0$ such that $X(\epsilon_0,s)$ is $X$ decomposable for any $s$ with the property $|s|<\epsilon$. Such epsilon exists by the theorem \ref{TH}. Now, to construct the normal chain starting from $X$ we define 

$$X_{jF+i}=X(\epsilon_0+j\epsilon/2,s_1,\ldots,s_i,0,\ldots,0)$$ where $s_i=\epsilon/2$ for all $i\in\{1,\ldots,F\}$. Let us double-check all the properties we require from the normal chain$\colon$

\begin{enumerate}
    \item Satisfied by the definition of $\epsilon_0$ and \ref{Lemma1}
    \item We have $X_{jL}=X(\epsilon_0+j\epsilon/2,0,\ldots,0)=(\epsilon_0+j\epsilon/2)X$.
    \item Satisfied by the definition of $X_{jF+i}$
    \item Since $(\epsilon_0+j\epsilon/2)X\subset X_{jF+i}$ and $j$ is unbounded, this property is satisfied.
    \item By the \ref{TH} this property is satisfied.
    \item Can be easily achieved by taking $\epsilon$ small enough.
\end{enumerate}

The most valuable property of normal chains is that moving along them commutes with taking the Pontryagin difference in the sense that will be discussed in the next section. 
\subsection{Chains of negative type.}
We will also need another type of chains whose elements possess similar properties. Since we have fixed the order of facets, we will denote the $i$-th facet of $X$ by $X^i$.

\begin{defin}\label{def2}

For a given convex polytope $X$ we will say that the chain of convex polytopes 

$$
 X_{JL}\subset X_{JL+1}\subset\ldots\subset X_{jL+i}\subset\ldots\subset X_{-L+L-1}\subset \epsilon_{0}X\subset X
$$
    is a negative normal chain if we have the following$\colon$

\begin{enumerate}
    \item for all $\epsilon$ such that $1>\epsilon>\epsilon_0$ we have $\mathbb{Z}(\epsilon X)=\mathbb{Z}(\epsilon_0 X)$.
    \item for some $J\in-\mathbb{N}$ and for some $s\in\{0,\epsilon_0/|J|\}$ for all $j\in\{J,\ldots,-1\}$ we have $X_{jF}=(\epsilon_0+js)X$. 
    \item for all $i\in\{1,\ldots, F-1\}$ and for all $j\in\{J,\ldots,-1\}$, $X_{jF+i}$ is given by the system $$\{x.\alpha_1\leq (\epsilon_0+js+s)\beta_1, \ldots,\ x.\alpha_i\leq (\epsilon_0+js+s)\beta_i,\ x.\alpha_{i+1}\leq (\epsilon_0+js)\beta_{i+1},$$ $$\ldots,\ x.\alpha_{F-1}\leq (\epsilon_0+js)\beta_{F-1}\}$$  
    \item for any $k\in\{JF,\ldots,-1\}$ and $i\in\{1,\ldots, F\}$ we have $X^{i}_{k}\ominus X^{i}=\varnothing$. 
    \item For all $k\in\{JF,\ldots,-1\}$ we have $\mathbb{Z}(X_{k+1}\setminus X_{k})$ is either empty or lies in a hyperplane.
\end{enumerate}

We denote negative normal chain $X_*^{-}$ for short.

\end{defin}

For any $J\in-\mathbb{N}$ we can achieve the desirable construction by taking $s\in(0,\epsilon_0/|J|)$ small enough so that (s) would be satisfied and

$$
X_{jF+i}=X(\epsilon_0+j\epsilon/2,s_1,\ldots,s_i,0,\ldots,0)
$$
where $s_i=s$ for all $i\in\{1,\ldots,F\}$. 

We only need to verify property (4). To do this, we prove the following proposition$\colon$

\begin{utver}\label{utver_0}
Let $X$ be defined by the system of inequalities $\{x.\alpha_i\leq \beta_i\}_{i=1}^F$ and let $s=(s_1,\ldots,s_F)$, $1>\epsilon, \epsilon_0>0$ be such that $X(\epsilon,s)\subset X(\epsilon_0,0)\subset X$. Then        
$$
X^{i}(\epsilon,s)\ominus X^{i}=\varnothing
$$
    
\end{utver}
\begin{proof}
    By the formula \ref{ominus} we know that $X^{i}(\epsilon,s)\ominus X^{i}$ is given by the set of inequalities

    $$
    x.\alpha_k\leq(\epsilon+s_k)\beta_k-\beta_k,\ k\neq i
    $$
    $$
    x.\alpha_i=(\epsilon+s_i)\beta_i-\beta_i
    $$
Since $\{\alpha_k\}_{k=1}^F$ generates $\mathbb{R}^{n}$ as a cone and $\epsilon+s_k-1<\epsilon_0-1<0$ this system is impossible to satisfy. Indeed, we can represent $x$ as a positive linear combination of $\{\alpha_k\}_{k=1}^F$, that is $x=\sum\lambda_k\alpha_k$ which means $x.x=\sum\lambda_k \alpha_k.x<0$            
\end{proof}

\subsection{\bf Chains of intermediate type.}

It only remains to understand what is happening between the terms $\epsilon_0 X$ and $X$ as we shift the facets of $\epsilon_0 X$ toward those of $X$. To do this we introduce the third part of the definition.

\begin{defin}\label{def3}

For a given convex polytope $X$ we will say that the chain of convex polytopes 

$$
 \epsilon_{0}X\subset X(\epsilon_0,s_1,\ldots,0)\subset\ldots\subset X(\epsilon_0,s_1,\ldots,s_i,0,\ldots,0)\subset\ldots\subset X(\epsilon_0,s_1,\ldots,s_{F-1},0)\subset X
$$
    is an intermediate normal chain if we have the following$\colon$

\begin{enumerate}
    \item for all $\epsilon$ such that $1>\epsilon>\epsilon_0$ we have $\mathbb{Z}(\epsilon X)=\mathbb{Z}(\epsilon_0 X)$.
    \item All $s_i=1-\epsilon_0$.
    \item for any $k\in\{JF,\ldots,-1\}$ and $i\in\{1,\ldots, F\}$ we have $X^{i}_{k}\ominus X^{i}=\varnothing$. 
    
\end{enumerate}

We denote intermediate normal chain $X_*^{0}\subset X$ for short.

\end{defin}

Note that the definition is not trivial although it looks very similar to what we worked on in the previous sections. The reason for that is that the order of deformations becomes important when we shift from $\epsilon_0 X$ to $X$ instead of the smaller homothety of $X$. 

\begin{exa}
This example shows us that for some order of deformation propert (3) might be violated.   

\begin{figure}[h!]
  \includegraphics[width=0.8\linewidth]{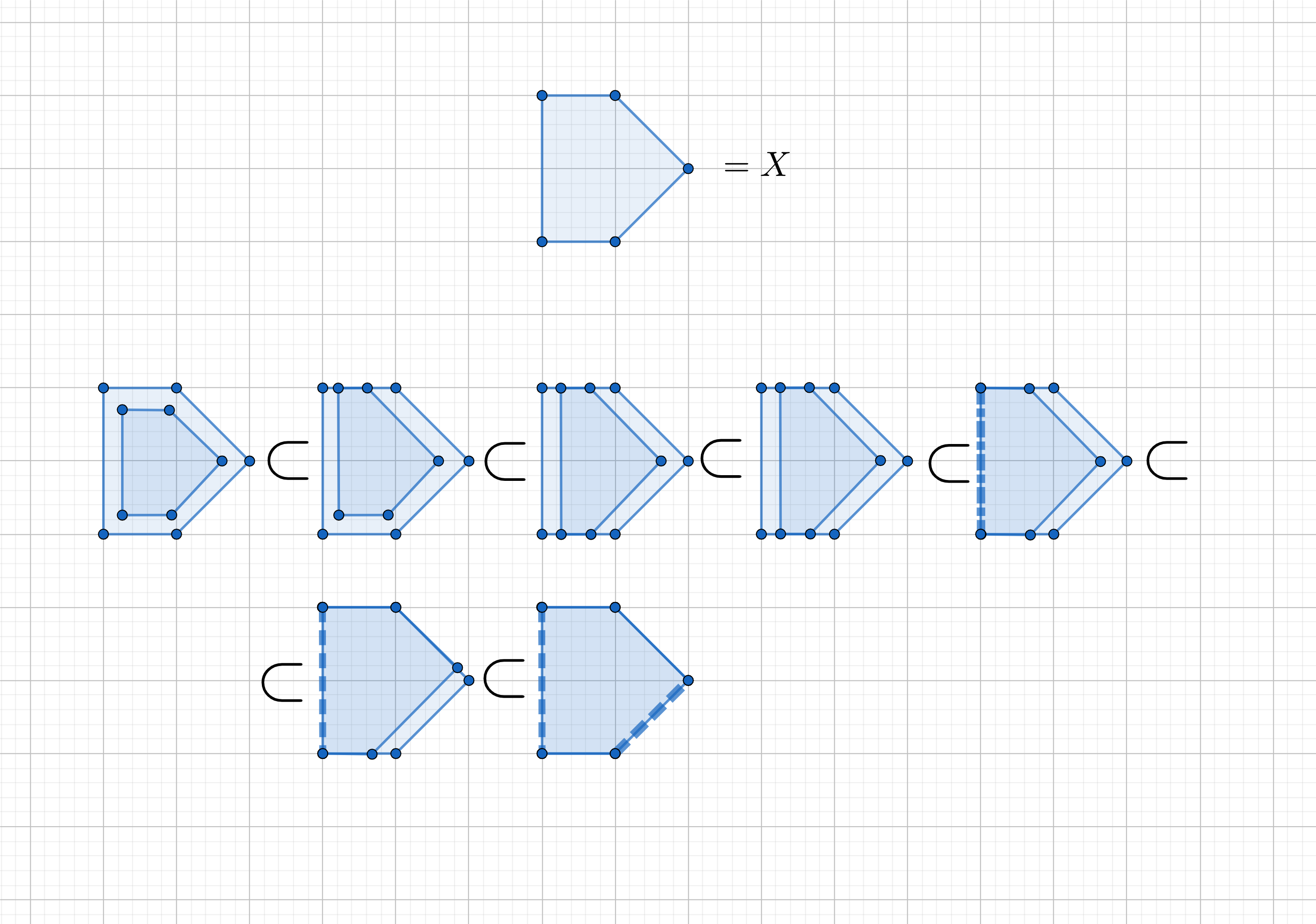}
  \caption{doted lines indicate steps at which (3) does not hold.}\label{fig3}
\end{figure}    
\end{exa}

We would like to show that the proper order always exists. 

\begin{utver}
    Given the data from the definition \ref{def3}, we can always find the order of deformation so that the property (3) holds along the chain.
\end{utver}
\begin{proof}
    It follows from \ref{utver_0} that $X^{i}_{k}\ominus X^{i}\neq\varnothing$ if and only if all facets conjugated to $X^{i}$ are already shifted. We can choose the order of deformations such that there is only one facet (obviously the final) for which $X^{i}_{k}\ominus X^{i}\neq\varnothing$. Consider the graph $\Gamma_X$ whose vertices are in correspondence with facets of $X$ and two vertices are connected by an edge if and only if the corresponding facets conjugate. We can locate vertices of $\Gamma_X$ in the corresponding vertices of dual polytope of $B$, for example. Now, we define an order $R_{\psi}$ on $\Gamma_X$ where $\psi\in\mathbb{R}^n$ a linear form. We will say that $v R_{\psi}u$ if and only if $\psi(v)\geq\psi(u)$. Let us pick $\psi$ such that all $\psi(v)$ are distinct for all $v\in\Gamma_X$. For this $\psi$ let us shift faces in order $R_\psi$. It is easy to see that a face of $v$ of $X$ is surrounded by shifted faces if and only if for any neighbors $u$ of the corresponding vertex $v$ we have $\psi(u)<\psi(v)$. But for given order it is possible only for the final vertex.
\end{proof}

\subsection{\bf Properties of Pontryagin difference related to normal chains and $\mathbb{Z}()$ operation.}\label{pc}

Just for simplicity of notation we give the following definition
\begin{defin}
    Let $X_*$ be a positive normal chain starting from $X$. We will call $\alpha\in\mathbb{R}^n$ direction of deformation in inclusion $X_m\subset X_{m+1}$ if $\alpha$ is perpendicular to the facet of deformation in this inclusion. 
\end{defin}

The goal of this section is to show that for a normal chain $X\subset X_*$ the following formula is satisfied$\colon$

$$
\mathbb{Z}((X_{m+1}\ominus X)\setminus(X_{m}\ominus X))=\mathbb{Z}((X_{m+1}\setminus X_{m})\ominus X(\alpha))
$$

\begin{lemma}\label{l10}
Let $X\subset X_*$ be a positive normal chain and $\alpha$ be the direction of deformation in the inclusion $X_m\subset X_{m+1}$. Then $(X_{m+1}\setminus X_m)\ominus X(\alpha)=(X_{m+1}\ominus X)\setminus(X_m\ominus X)$.   
\end{lemma}

\begin{proof}
    \begin{enumerate}
        \item $\subset$ part$\colon$Let $x\in(X_{m+1}\setminus X_m)\ominus X(\alpha)$. By definition $x+X\not\subset X_m$, so $x\notin X_m\ominus X$. Since $X_{m+1}$ is an element of a normal chain, we deduce that $x$ can be extended to an element of $X_{m+1}\ominus X$ so $x\in(X_{m+1}\ominus X)\setminus(X_m\ominus X)$.
        \item $\supset$ part$\colon$ Let $x\in(X_{m+1}\ominus X)\setminus(X_m\ominus X)$. We can describe $X_m$ and $X_{m+1}$ by two sets of inequalities which vary by only one single inequality $x.\alpha\leq(\epsilon_0+js)\beta$ for $X_{m}$ and $x.\alpha\leq(\epsilon_0+js+s)\beta$ for $X_{m+1}$. From $x\in(X_{m+1}\ominus X)\setminus(X_m\ominus X)$ we know that $X$ satisfies all of the ineqaulities except the last one. Since $\alpha$ is perpendicular to $X(\alpha)$, we can only have a few possibilities$\colon$
        \begin{enumerate}
            \item $X(\alpha)\subset X_m$
            \item $X(\alpha)$ has no points in $X_m$ at all.
        \end{enumerate}
        In the first case $X\subset X_{m}$ which contradicts the condition $x\notin(X_m\ominus X)$. In the second case $X(\alpha)\subset X_{m+1}\setminus X_m$.
    \end{enumerate}
    
\end{proof}

To prove the main result of the next section we must make the final preparatory steps. In general, elements of $C$ will not be integral polytopes, but we are going to use them as supports for some polynomial equations. To do this we prove the following formula$\colon$

\begin{lemma}\label{l12}
   Let $X\subset\mathbb{Z}^n$ be a convex polytope and $X_m\subset X_{m+1}$ be like in the lemma \ref{l10}. Then $\mathbb{Z}((X_{m+1}\ominus X)\setminus (X_m\ominus X))=\mathbb{Z}(X_{m+1}\setminus X_m)\ominus X(\alpha)$.
\end{lemma}

\begin{proof}
    By lemma \ref{l10} we have $(X_{m+1}\setminus X_m)\ominus X(\alpha)=(X_{m+1}\ominus X)\setminus(X_m\ominus X)$. From this we tautologically obtain $$\mathbb{Z}((X_{m+1}\ominus X)\setminus(X_m\ominus X))=\mathbb{Z}((X_{m+1}\setminus X_m)\ominus X(\alpha))$$
We are going to show that $\mathbb{Z}((X_{m+1}\setminus X_m)\ominus X(\alpha))=\mathbb{Z}(X_{m+1}\setminus X_m)\ominus X(\alpha)$. Let $Y\subset\mathbb{R}^n$ be a convex subset and $B\subset\mathbb{Z}^n$ also be convex. Then $\mathbb{Z}(Y\ominus B)=\mathbb{Z}(Y)\ominus B$. Indeed$\colon$

\begin{enumerate}
    \item $\subset$ part$\colon$ Let $x\in\mathbb{Z}(Y\ominus B)$ then $x\in\mathbb{Z}^n,\ x+B\subset Y$. Assume that $x+B\not\subset\mathbb{Z}(Y)$. That means we have a vertex $y$ of $B$ such that $x+y\notin\mathbb{Z}(Y)$. But $x+y\in\mathbb{Z}^n$ and if $x+y\in Y$ then $x+y\in\mathbb{Z}(Y)$.
    \item $\supset$ part$\colon$ Let $x\in\mathbb{Z}(Y)\ominus B$ then $x+B\subset\mathbb{Z}(Y)\subset Y$ and $x\in\mathbb{Z}(Y\ominus B)$.
\end{enumerate}
Finally, we apply this formula to the case $Y=X_{m+1}\setminus X_m$ and $B=X(\alpha)$.
\end{proof}

Condition for $X$ to have integral vertices is important. See figure \ref{y8}.

\begin{figure}[h!]
  \centering
     \includegraphics[width=0.5\linewidth]{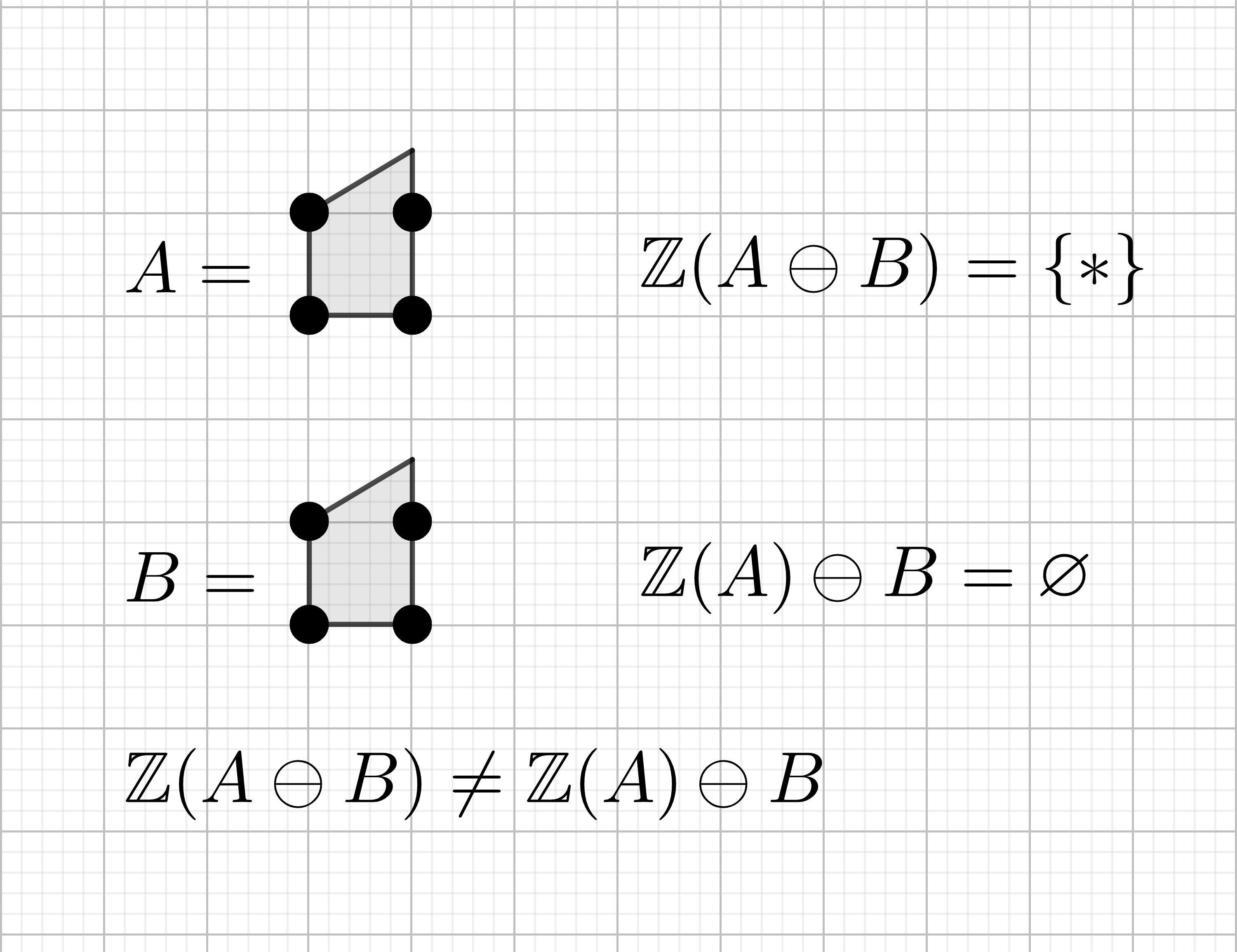}
    \caption{example that shows why $X$ must be integral}
\label{y8}
\end{figure}

\begin{lemma}\label{l11}
    Let $X\subset\mathbb{Z}^n$ be a convex polytope. Let $X\subset X_*$ be a positive normal chain. Then $\mathbb{Z}(X_{m+1}\ominus X)\setminus\mathbb{Z}(X_m\ominus X)=\mathbb{Z}((X_{m+1}\ominus X)\setminus (X_m\ominus X))$. Also, we have $\mathbb{Z}(X_m\ominus X)\subset\mathbb{Z}(X_{m+1}\ominus X)$.
\end{lemma}

\begin{proof}
\begin{enumerate}
    \item $\subset$ part$\colon$ Let $x\in\mathbb{Z}(X_{m+1}\ominus X)\setminus\mathbb{Z}(X_m\ominus X)\Rightarrow x\in X_{m+1}\ominus X,\ x\in\mathbb{Z}^n$. Assume that $x\in X_m\ominus X$, then $x\in\mathbb{Z}(X_m\ominus X)$ since $x$ is integral. This is not possible. We deduce that $x\in(X_{m+1}\ominus)\setminus(X_m\ominus X)$. Since $x\in\mathbb{Z}^n$ we conclude that $x\in\mathbb{Z}((X_{m+1}\ominus X)\setminus(X_m\ominus X))$.
    \item $\supset$ part$\colon$ Let $x\in\mathbb{Z}((X_{m+1}\ominus X)\setminus(X_m\ominus X))\Rightarrow x\in\mathbb{Z}^n,\ x+X\subset X_{m+1}, x+X\not\subset X_m$. Thus, $x\in\mathbb{Z}(X_{m+1}\ominus X)$. Assume that $x\in\mathbb{Z}(X_m\ominus X)=\mathbb{Z}(X_m)\ominus X$. Then $x+X\subset \mathbb{Z}(X_m)\subset X_m$. Contradiction.
\end{enumerate} 
For the second part, consider $x+X\subset X_m\subset X_{m+1}\Rightarrow X_m\ominus X\subset X_{m+1}\ominus X\Rightarrow\mathbb{Z}(X_m\ominus X)\subset\mathbb{Z}(X_{m+1}\ominus X)$.
\end{proof}

\section{\bf Proof of theorem \ref{THM}.}

\subsection{\bf Computing $\dim\mathcal{V}^{f}_{A}$ in dimension $n$ for $2$ copies of $B$.}\label{ga3}

We dedicate this subsection to the proof of theorem \ref{THM} for $3$ bodies $(A, B, B)$. We will use widely all the material we have collected in two previous purely technical sections.

\begin{theor}\label{THM_0}
\indent
Let $A, B$ be polytopes in $\mathbb{Z}^n$ such that the last one is not a segment. Let $A+v\subset(t+1)B$ for some non-negative real $t$ and $v\in\mathbb{Z}^n$. Then, for any r-general couple of polynomials $f_1,f_2$ supported at $B$ and any combination $g=c_1 f_1+c_2 f_2$ supported at $A$, we can represent $g$ in the form $g=s_1 f_1+s_2 f_2$ where all $s_i$ are supported at $\mathbb{Z}(tB)$.
\end{theor}

Let $C$ be a polytope. As in the example \ref{mex} we define the following map$\colon$
$$
\phi^C\colon\mathbb{C}^{C}\oplus\mathbb{C}^{C}\rightarrow\mathbb{C}[{\bf{x}}^{\pm1}]
$$
$$
(c_1, c_2)\mapsto c_1 f_1+c_2 f_2
$$

In the space $\mathbb{C}^{C}$ we define a linear subspace $W_A^C$ by the condition $\supp(c_1 f_1+c_2 f_2)\subset A$. In this notation we have the following formula which actually works for any number of bodies

\begin{equation}\label{eq_0}
\dim\mathcal{V}_A^{C, f}=\dim W_A^C-\dim\ker\phi^C.\quad 
\end{equation}

We will use notation 

$$
\delta\dim W^{C_i}_A=\dim W^{C_{i+1}}_A-\dim W^{C_{i}}_A
$$

Now, choose $C$ as in the theorem \ref{THM_0}. To show that $\mathcal{V}_A^{C, f}=\mathcal{V}_A^f$ we will apply our study of normal chains. 

First, we describe how both terms of \ref{eq_0} changes when we add extra points to $C$.

\begin{lemma}\label{Ker}
    In the above notation, $\dim\ker\phi^{C}=\vert\mathbb{Z}(C)\ominus B\vert$.
\end{lemma}
\begin{proof}
    For r-general $f=(f_1, f_2)$ from $c_1 f_1+c_2 f_2=0$ we obtain $(c_1,\ c_2)=\mu(f_2,\ -f_1)$ where $\mu\in\mathbb{C}^{\mathbb{Z}(C)\ominus B}$. 
\end{proof}

The next lemma shows us that actually $\dim W_A^{C}$ cannot change to fast along normal chains. That is why we use them.

\begin{lemma}\label{LC}
    If $C_{i}\subset C_{i+1}$ is an inclusion in normal chain such that $A\subset\mathbb{Z}(C_i+B)$ then 
    
    $$
    \delta\dim W^{C_i}_A\leq|\mathbb{Z}(\delta C_i)\ominus B(\alpha)|.
    $$
    where $\alpha$ is the direction of deformation in the inclusion $C_i\subset C_{i+1}$.
\end{lemma}
\begin{proof}
    Let $W_A^{C_i}$ be defined by the system $\Omega_i$ of linear equations such that we enumerate columns of $\Omega_i$ by points of $\mathbb{Z}(C_i)$ and rows by linear conditions obtained from the property that $c_1 f_1+c_2 f_2$ has no monomials outside $A$. Thus $\dim W^{C_i}_A=2|\mathbb{Z}(C_i)|-\rk\Omega_i$. We use similar notation for $\Omega_{i+1}$. Under this agreement we can see that $\Omega_{i+1}$ has the following form$\colon$

\begin{equation}\label{eqM}
\Omega_{i+1}=\left[\begin{array}{c|c c} 
	\Omega_i & *\\ 
	\hline 
	0 & \delta\Omega
\end{array}\right]
\end{equation}
    where $\delta\Omega$ is the matrix in the intersection of new columns related to the points in $\mathbb{Z}(C_{i+1}\setminus C_i)$ and conditions given by new monomials in $(c_1+\delta c_1) f_1+(c_2+\delta c_2) f_2$. Since all new monomials $\mathbb{Z}(\delta C_i)=\mathbb{Z}(C_{i+1}\setminus C_i)$ lie in the common hyperplane we can group them up. After this we obtain a single algebraic equation$\colon$

$$
c_1\vert_{\mathbb{Z}(\delta C_i)} f_1\vert_{B(\alpha)}+c_2\vert_{\mathbb{Z}(\delta C_i)} f_2\vert_{B(\alpha)}=0
$$
    We deduce from it that 
    $$
    (\delta c_1, \delta c_2)=\lambda(f_2\vert_{B(\alpha)},\ -f_1\vert_{B(\alpha)})
    $$
    where $$\lambda\in\mathbb{C}^{\mathbb{Z}(\delta C_i)\ominus B(\alpha)}.$$
    By the dimension theorem from linear algebra we have $$2\vert\mathbb{Z}(\delta C_i)\vert-\rk\delta\Omega=\mathbb{Z}(\delta C_i)\ominus B(\alpha).$$ 

    By the inequality $\rk\Omega_{i+1}\geq\rk\Omega_{i}+\rk\delta\Omega$ we have $\delta\rk\Omega_i=\rk\Omega_{i+1}-\rk\Omega_{i}\geq\rk\delta\Omega$.
and finally we have 

$$
\delta\dim W^{C_i}_A=\delta(2|\mathbb{Z}(C_i)|-\rk\Omega_i)=2\vert\mathbb{Z}(\delta C_i)\vert-\delta\rk\Omega_i\leq2\vert\mathbb{Z}(\delta C_i)\vert-\rk\delta\Omega=$$ $$=|\mathbb{Z}(\delta C_i)\ominus B(\alpha)|.
$$
    
\end{proof}
Since we have just introduced $\Omega_i$ let us use it to prove that $\dim\mathcal{V}^f_A$ does not depend on the choice of $f$ for generic $f$. We will do it under the assumption of the theorem \ref{THM} that we will prove independently. 
\begin{utver}\label{Generic}
    Let $(A,B_1,\ldots,B_k)$ be a tuple of convex subsets from $\mathbb{Z}^n$ such that all $B_i$ are the same and $\dim B>k$. Then $\dim\mathcal{V}^f_A$ does not in general depend on the choice of $f\in\mathbb{C}^B$.
\end{utver}

\begin{proof}
    In this proof, given a square matrix $Q$ we will denote its rank by $\rk Q$ and its size by $\dim Q$. Let $\Omega_C(f)$ be the matrix of the system of linear equations imposed on $c_i$ and given by the conditions $\supp(c_1 f_1+\ldots+c_k f_k)\subset A$ where all $c_i$ are from $\mathbb{C}^C$. And let $\Omega_C^0(f)$ be the matrix of the linear system $c_1 f_1+\ldots+ c_k f_k=0$ under the same restrictions on $c_i$. Then we have the formula

$$\dim\mathcal{V}^{C,f}_A=(k|\mathbb{Z}(C)|-\rk\Omega_C(f))-(k|\mathbb{Z}(C)|-\rk\Omega_C^0(f)).$$

After the cancellation we have

$$
\dim\mathcal{V}^{C,f}_A=\rk\Omega_C^0(f)-\rk\Omega_C(f).
$$

For fixed $C$ let us consider $\Omega_C^0(f)$ and $\Omega_C(f)$. Let 

$$
m^0_1(f),\ m^0_2(f),\ldots,\ m_i^0(f),\ldots
$$

$$
m_1(f),\ m_2(f),\ldots,\ m_i(f),\ldots
$$

be sets of all minors of $\Omega_C^0(f),\ \Omega_C(f)$ in such order that $\dim m_i^0(f)\geq\dim m_j^0(f)$ and $\dim m_i(f)\geq\dim m_j(f)$ for any $i\geq j$. Let $I$ and $J$ be the smallest integers such that $\det m^0_I(f)$ and $\det m_J(f)$ are not identically zero. We will denote their dimensions by $M^0(C)$ and $M(C)$ since they depend only on the choice of $C$. Let $X, Y\subset\mathbb{C}^B$ be open subsets $X=\{\det m_I(f)\neq0\}$ and $Y=\{\det m_J(f)\neq0\}$. We can see that for all $f\in X\cap Y$ we have $\dim\mathcal{V}^{C,f}_A=\rk\Omega_C^0(f)-\rk\Omega_C(f)=\rk m^0_I(f)-\rk m_J(f)=M^0(C)-M(C)$ which does not depend on the choice of an element from the Zariskii open subset $X\cap Y$. Since we see that for almost all elements in $\mathbb{C}^B$ the value of $\mathcal{V}_A^{C,f}$ does not depend on $f$, let us denote it by $D(C)$, for short. We have the property that $C_1\subset C_2$ implies $D(C_2)\geq D(C_1)$ by the definition of $\mathcal{V}_A^{C,f}$ and by the fact that the intersection of two Zarisskii open subsets is again open. For given $C$ we will denote by $U(C)$ the open subset defined above which is characterized by the property that $f\in U(C)$ implies $\dim\mathcal{V}_A^{C,f}=D(C)$.

Since $D(C)\leq|A|$, we will have $D(C_m)=D(C)$ for some $C_m$ and for any $C$ such that $C_m\subset C$. From \ref{l21} we know that there exists some open dense subset $Y\subset\mathbb{C}^B$ such that $f\in Y$ are all r-general. From \ref{THM} it follows that there exists $C_0$ such that $\dim\mathcal{V}_A^{C_0,f}=\dim\mathcal{V}_A^f$. So for all $f\in Y\cap U(C_0)$ we have the same $\dim\mathcal{V}_A^{f}$.
\end{proof}
    We move on to our normal topic.
\begin{theor}\label{T}
    Let $A, B\subset\mathbb{Z}^n$ and $f=(f_1, f_2)$ be a r-general couple of Laurent polynomials whose Newton polytope is $B$. Let $C_*=\{C^{-}_*\subset C^{0}_*\subset B\subset C^{+}_*\}$ be a normal chain. Let $C_I$ be the least possible member of $C_*$ such that $A\subset\mathbb{Z}(C_I+B)$. Then $\dim\mathcal{V}_A^{C_i,f}=\dim\mathcal{V}_A^{C_I,f}=\dim\mathcal{V}_A^f$ for all $i\geq I$.
\end{theor}
Taking the least possible $t$ such that $A\subset C_I+B\subset(t+1)B$ we obtain the same result as in the theorem \ref{THM_0}
\begin{proof}
    Let $\alpha$ be the direction of deformation in the inclusion $C_i\subset C_{i+1}$. By lemma \ref{LC} we have

\begin{equation}\label{Van}
\delta D_i=\dim\mathcal{V}_A^{C_{i+1},f}-\dim\mathcal{V}_A^{C_i,f}\leq|\mathbb{Z}(\delta C_i)\ominus B(\alpha)|-\delta\dim\ker\phi^{C_i}.
\end{equation}
    
    \begin{enumerate}
        \item If $C_i\subset C_{i+1}$ is in the positive branch of $C_*$, 
by the formula obtained in lemma \ref{Ker} we have
$$\delta\dim\ker\phi^{C_i}=|\mathbb{Z}(C_{i+1})\ominus B|-|\mathbb{Z}(C_{i})\ominus B|=|\mathbb{Z}(C_{i+1}\ominus B)|-|\mathbb{Z}(C_{i}\ominus B)|.$$

The second equality is from \ref{l12}. If $|\mathbb{Z}(C_{i+1})|=|\mathbb{Z}(C_{i})|$ then we obviously have $\delta D_i=0$. Otherwise, 
$$|\mathbb{Z}(C_{i+1}\ominus B)|-|\mathbb{Z}(C_{i}\ominus B)|=|\mathbb{Z}((C_{i+1}\ominus B)\setminus (C_{i}\ominus B))|$$

by lemma \ref{l11}. By lemma \ref{l12} $$\mathbb{Z}((C_{i+1}\ominus B)\setminus (C_i\ominus B))=\mathbb{Z}(\delta C_i\ominus B(\alpha))=\mathbb{Z}(\delta C_i)\ominus B(\alpha).$$ We get $\delta\dim\ker\phi^{C_i}=|\mathbb{Z}(\delta C_i)\ominus B(\alpha)|$. And $0\leq\delta D_i\leq\vert\mathbb{Z}(\delta C_i)\ominus B(\alpha)\vert-\vert\mathbb{Z}(\delta C_i)\ominus B(\alpha)\vert=0$. We conclude that $\delta D_i=0$ for all $i$ in the normal chain starting from $I$.   

    \item If $C_i\subset C_{i+1}$ is in the negative branch of $C_*$ then both terms in \ref{Van} are zero.
    \item If $C_i\subset C_{i+1}$ are in the intermediate branch, then just like in the previous part, both terms are zero unless $C_{i+1}\neq B$. In this partial case when $C_{i+1}=B$ we have 

$$
\delta D_i\leq|\mathbb{Z}(\delta C_i)\ominus B(\alpha)|-\delta\dim\ker\phi^{C_i}=|\mathbb{Z}(B(\alpha)\ominus B(\alpha)|-(1-0)=1-1=0.
$$
    
\end{enumerate}

\end{proof}

We emphasize an important partial case, or modification, of theorem \ref{THM} when $A$ is an element of a normal chain itself, lemma \ref{LEMMA}. It will play crucial role in the proof of theorem \ref{THM} in complete generality. Namely, we can construct a normal $C_*$ chain so that $C_i\ominus B$ is another member of the same normal chain or empty.

\begin{lemma}\label{LEMMA}
    There exists a normal chain $C_*^-\subset C_*^0\subset B\subset C_*^+$ such that either
\begin{enumerate}
    \item $C_{jF+i}\ominus B=C_{(j-m)F+i}$ for some integer $m$ or 
    \item $\mathbb{Z}(C_{jF+i}\ominus B)\subset B$
\end{enumerate}
We will call such normal chains convenient.    
\end{lemma}
\begin{proof}
Let us represent $s$ in the form $s=1/N$ for some $N\in\mathbb{N}$ which we will define later, $s_k=s$ for all $k\in\{1,\ldots, i\}$ and 

$$
C_{jF+i}=B(\epsilon_0+js,s_1,\ldots,s_i,0\ldots,0)
$$

for positive integral $j$.\\

By T2.3 from \cite{G} we have that 

$$
C_{jF+i}\ominus B=\{\alpha_i.x\leq(\epsilon_0+js+s_k)\beta_k-h_B(\alpha_k)\}=$$

$$=\{\alpha_k.x\leq(\epsilon_0+js-1+s_k)\beta_k\}=B(\epsilon_0+js-1,s_1,\ldots,s_i,o,\ldots,0)
$$

Since $s=1/N$, $js-1=j/N-1=(j-N)/N=\bar{j}s$, and 

$$
B(\epsilon_0+js-1,s_1,\ldots,s_i,o,\ldots,0)=C_{(j-N)F+i}
$$

We can see that for $j\geq N$, $C_{jF+i}\ominus B$ is a member of the normal chain and $$\mathbb{Z}(C_{jF+i}\ominus B)=\mathbb{Z}(\{\alpha_k.x\leq(\epsilon_0+js-1+s_k)\beta_k\})\subset\mathbb{Z}(\epsilon_0)B=\mathbb{Z}(B).$$

when $js-1=(j-N)/N<0$. We can choose $N$ large enough so that $C_{jF+i}$ defined above indeed group up into a normal chain. That is possible by theorem \ref{TH}. 

\end{proof}

We end this section with another technical fact which complete the topic.

\begin{defin}
    Let $C\ominus^1 B=C\ominus B$, we put $C\ominus^k B=(C\ominus^{k-1} B)\ominus B$ for any natural $k$ and also for convenience we put $C\ominus^{-1}B=C+B$, $C\ominus^{0}B=C$.
\end{defin}

\begin{lemma}\label{LEM}
    Let $C_*^-\subset C_*^0\subset B\subset C_*^+$ be a convenient normal chain. Then for any $C\in C_*$ either
    \begin{enumerate}
        \item $C\ominus^k B\in C_*$ or
        \item $\mathbb{Z}(C\ominus^k B)\subset B$
    \end{enumerate}
\end{lemma}

\begin{proof}
    Since $C_*$ is convenient, by lemma \ref{LEMMA} we have that $C\ominus B$ is an element of the same normal chain or includes only those integral point which are already in $B$. If necessary, in both cases we can repeat the reasoning and finish the proof.
\end{proof}

\section{\bf Proof for $k>3$.}\label{MR}

In this final section we are going to use the partial case $k=3$ as a basis of induction. Let $A+v\subset\mathbb{Z}((1+t)B)$ for some real $t>1$ and $v\in\mathbb{Z}^n$ where $A, B\subset\mathbb{Z}^n$ are convex polytopes and $k\geq3$. We consider $k$ copies of $B$ in the tuple. We will use a convenient normal $C_*$ chain related to $B$. 

We give a brief reminder of the definition of Koszul complex and some of its properties. For a given ring (we will only use the multivariate ring $R=\mathcal{L}[{\bf{x}, \bf{x}^{-1}}]$ of Laurent polynomials, for simplicity) and a map of evaluation $ev\colon R^{k}\rightarrow R$, $e_i\mapsto f_i$, Koszul complex defined on this data is 

$$
0\rightarrow\bigwedge^k R^k\rightarrow\bigwedge^{k-1} R^k\rightarrow\ldots\rightarrow\bigwedge^{1} R^k\rightarrow R\rightarrow R/(f_1,\ldots, f_k)\rightarrow0
$$

where each differential is formally defined by the rule 

$$
x_1\wedge\ldots\wedge x_r\mapsto\sum_{i=1}^{r}(-1)^{i+1}ev(x_i) x_1\wedge\ldots\wedge\hat{x}_i\wedge\ldots\wedge x_r
$$

In different notation, we can represent Koszul complex in the form 

$$
0\rightarrow R\xrightarrow[]{\phi_0}R^{{k}\choose{1}}\xrightarrow[]{\phi_1}\ldots\xrightarrow[]{\phi_{r-1}}R^{{k}\choose{r}}\xrightarrow[]{\phi_{r}}\ldots\xrightarrow[]{\phi_{k-2}}R^{{k}\choose{k-1}}\xrightarrow[]{\phi_{k-1}}R\xrightarrow[]{\phi_{k}}R/(f_1,\ldots, f_k)
$$

The meaning of $\phi_k$ is just a factorisation by the ideal $(f_1,\ldots, f_k)$ and $\phi_{k-1}(s_1,\ldots, s_k)=s_1 f_1+\ldots+s_k f_k$.

We would like to replace each term by a suitable finite dimensional vector space, so that exactness of the complex would be preserved. That will allow us to estimate dimensions in the formula 

$$
\dim\mathcal{V}_A^{C, f}=\dim W_A^C-\dim\ker\phi^C.
$$

It is easy to verify the following property of $\phi_r\colon$

\begin{utver}\label{utver}
For $r<k-1$, in matrix notation, each row of $\phi_r$ consists of zeroes and $\pm f_i$ such that the total number of non-zero entries is smaller than $k$. In each row all non-zero entries correspond to different $f_i$.
\end{utver}
\begin{proof}
    We start from the second statement. Let $e_I=\bigwedge^r_{k=1} e_{i_k}$. If $\phi_r(e_I)$ has $f_l$ in the first row (wlog) we conclude that $e_I=\pm e_{i_l}\wedge\omega$, where $\omega$ is the first standard basis vector in $\bigwedge^{r-1} R^k$. If $f_l$ appears twice in the first row, that means $e_I=\pm e_J$ for two different index subsets $I$ and $J$, which is impossible. 

    Let us say, we can find all $\{f_i\}_{i=1}^k$ in the first row. That means, multiplying $\omega$ by any $e_l$ we would always obtain a nonzero tensor $e_l\wedge\omega$. That is only possible if $\deg\omega=0$ and thus $r=k-1$.
\end{proof}

Let $C_*$ be a convenient normal chain and $C\in C_*$. We define $\mathcal{L}_{k-i-1}=\mathbb{C}^{C\ominus^iB}$, $\mathcal{L}_{k-i-1}=\{0\}$ if $C\ominus^iB=\varnothing$ for $i\in\{-1,\ldots, k-1\}$ and $\psi_{k-i}=\phi_{k-i}\vert_{\mathcal{L}_{k-i}}$ for $i\in\{0,\ldots, k-1\}$. We will denote the following construction by $\mathcal{K}(B,C,f)$.
$$
0\xrightarrow[]{}\mathcal{L}_{0}\xrightarrow[]{\psi_0}\mathcal{L}_{1}^{{k}\choose{1}}\xrightarrow[]{\psi_1}\ldots\xrightarrow[]{\psi_{k-i-1}}\mathcal{L}_{k-i}^{{k}\choose{k-i}}\xrightarrow[]{\psi_{k-i}}\ldots\xrightarrow[]{\psi_{k-3}}\mathcal{L}^{{k}\choose{k-2}}_{k-2}\xrightarrow[]{\psi_{k-2}}\mathcal{L}^{{k}\choose{k-1}}_{k-1}\xrightarrow[]{\psi_{k-1}}\mathcal{L}^{{k}\choose{k}}_{k}\xrightarrow[]{}0
$$

\begin{theor}\label{THM2}
We claim that $\mathcal{K}(B,C,f)$ is a well defined chain complex. For r-general $(f_1,\ldots, f_k)$, $\mathcal{K}(B,C,f)$ is exact in each term except maybe in the $k$-th provided the statement of the theorem \ref{THM} is true for $k-1$ copies of $B$. 
\end{theor}

\begin{proof}
    \begin{enumerate}
        \item First, we show that the differentials are well defined. Let $s=\{s_j\}^{{k}\choose{k-i}}_{j=1}\in\mathcal{L}^{{k}\choose{k-i}}_{k-i}$, then for any $l\in\{1,\ldots,k\}$ we have $\supp (s_j f_l)\subset (C\ominus^{i+1}B)+B\subset C\ominus^{i}B\Rightarrow\psi_{k-i}(s)\in\mathcal{L}_{k-i+1}$.
        \item Chain property follows from the definition of $\phi_i$.
        \item Now, we show that for $s=\{s_j\}^{{k}\choose{k-i}}_{j=1}\in\mathcal{L}^{{k}\choose{k-i}}_{k-i}$ such that $s\in\ker\psi_{k-i}$, we can find a representation of $s$ as an image of some element from $\mathcal{L}_{k-i-1}$. Since $s\in\ker\psi_{k-i}$ we know that $s\in\ker\phi_{k-i}$ automatically. So, from the original Koszul complex we can find an element $s_0\in R^{{k}\choose{k-i-1}}$ such that $s=\phi_{k-i-1}(s_0)$. Using proposition \ref{utver} we can represent $s_i$ in the form$\colon$

$$
s_j=\sum_{r=1}^{k-1}c^j_rf_{r}
$$

Where $c^j_r$ are all from $R$. We can see that $\supp s_j\subset C\ominus^{i+1}B$ and thus by theorem \ref{THM} for $k-1$ copies of $B$ and by lemmata \ref{LEMMA}, \ref{T}, $c_r^j$ can be taken from $\mathcal{L}_{k-i-1}$.
        
    \end{enumerate}
\end{proof}

Now, consider the system of linear equations 

$$
c_1 f_1 + \ldots + c_k f_k = 0,\quad c_i\in\mathbb{C}^{tB}
$$
Since $(f_1,\ldots, f_k)$ is an ideal in the polynomial ring generated by a regular sequence, we can apply theorem \ref{THM2} to compute $\dim\im\psi_{k-2}(C)$ for given $C$ which is equal to the dimension of the set of solutions since $\mathcal{K}(B,C,f)$ is exact in this term. Consider another complex, which is exact in each term 

$$
0\xrightarrow[]{}\mathcal{L}_{0}\xrightarrow[]{\psi_0}\mathcal{L}_{1}^{{k}\choose{1}}\xrightarrow[]{\psi_1}\ldots\xrightarrow[]{\psi_{k-i-1}}\mathcal{L}_{k-i}^{{k}\choose{k-i}}\xrightarrow[]{\psi_{k-i}}\ldots\xrightarrow[]{\psi_{k-3}}\mathcal{L}^{{k}\choose{k-2}}_{k-2}\xrightarrow[]{\psi_{k-2}}\im\psi_{k-2}\xrightarrow[]{}0
$$

By exactness, its' Euler characteristics is zero, so we have 

$$
-\dim\im(\psi_{k-2})+\sum_{j=0}^{k-2}(-1)^{j}{{k}\choose{j}}\dim\mathcal{L}_{j}=0
$$

By the definition of $\mathcal{L}_j$ we can reformulate the last term

$$
-\dim\im(\psi_{k-2})+\sum_{j=0}^{k-2}(-1)^{j}{{k}\choose{j}}|\mathbb{Z}(C\ominus^{k-j-1} B)|=0
$$

In the end we have

\begin{equation}
    \dim\im(\psi_{k-2})=\sum_{j=0}^{k-2}(-1)^{j}{{k}\choose{j}}|\mathbb{Z}(C\ominus^{k-j-1} B)|.
\end{equation}

We emphasize that all computations were made under the assumption that we have already proven theorem \ref{THM} for the induction step $k-1$. Now we will expand all steps of the induction to prove \ref{THM} for arbitrary $k$.

Formula from the previois chapter

$$
\delta D_i=\dim\mathcal{V}_A^{f.C_{i+1}}-\dim\mathcal{V}_A^{f.C_i}=\delta\dim W_A^{C_i}-\delta\dim\ker\phi^{C_i}
$$
holds in any dimension. We want to show that $\delta D_i\leq0$ in each term of the normal chain. Let us observe the first term $\delta\dim W^{C_i}_A$.

\begin{equation}\label{eqL}
\Omega_{i+1}=\left[\begin{array}{c|c c} 
	\Omega_i & *\\ 
	\hline 
	0 & \delta\Omega
\end{array}\right]
\end{equation}

Augmented matrix for the linear system $\supp(c_1 f_1+\ldots+c_{k-1}f_{k-1})\subset A$ has the same structure as it had in the previous section, in the case $k=3$ (see \ref{eqL}). Thus, $$\delta\dim W^{C_i}_A=\delta((k-1)|\mathbb{Z}(C_i)|-\rk\Omega_i)=(k-1)|\mathbb{Z}(\delta C_i)|-\delta\rk\Omega_i\leq$$

$$
\leq(k-1)|\mathbb{Z}(\delta C_i)|-\rk\delta\Omega_i=\dim\im\psi_{k-2}(\delta C_i).
$$
Where $\delta C_i$ indicates that we consider Koszul complex related to $B(\alpha)$ and $\dim\im\psi_{k-2}(\delta C_i)$ is equal to the dimension of the space of solutions of the system 
\begin{equation}\label{eqL_0}
c_1\vert_{\mathbb{Z}(\delta C_i)} f_1\vert_{B(\alpha)}+\ldots+c_{k-1}\vert_{\mathbb{Z}(\delta C_i)} f_{k-1}\vert_{B(\alpha)}=0
\end{equation}

Now, $\dim\ker\phi^{C_i}=\dim\im\psi_{k-2}(C_i)$ and we have 

$$
\delta D_i\leq\dim\im(\psi_{k-2})(\delta C_i)-\delta(\dim\im(\psi_{k-2})(C_i))=
$$

$$
=\sum_{j=0}^{k-2}(-1)^{j}{{k}\choose{j}}|\mathbb{Z}(\delta C_i\ominus^{k-j-1} B(\alpha))|-\delta(\sum_{j=0}^{k-2}(-1)^{j}{{k}\choose{j}}|\mathbb{Z}(C_i\ominus^{k-j-1} B)|)=
$$

$$
=\sum_{j=0}^{k-2}(-1)^{j}{{k}\choose{j}}|\mathbb{Z}(\delta C_i\ominus^{k-j-1} B(\alpha))|-\sum_{j=0}^{k-2}(-1)^{j}{{k}\choose{j}}\delta|\mathbb{Z}(C_i\ominus^{k-j-1} B)|=
$$

$$
=\sum_{j=0}^{k-2}(-1)^{j}{{k}\choose{j}}(|\mathbb{Z}(\delta C_i\ominus^{k-j-1} B(\alpha))|-\delta|\mathbb{Z}(C_i\ominus^{k-j-1} B)|)=0
$$

Since each $C_i$ is an element of a normal chain and thus $\mathbb{Z}(\delta C_i\ominus^j B(\alpha))=\delta\mathbb{Z}(C_i\ominus^j B)$. By induction, the statement is proven.

\end{document}